\documentclass[review,onefignum,onetabnum]{siamart171218}

\usepackage{lipsum}
\usepackage{color}
\usepackage{amsfonts}
\usepackage[leqno]{amsmath}
\usepackage{amssymb}
\usepackage{graphicx}
\usepackage{epstopdf}
\usepackage{algorithmic} \usepackage{subcaption,caption}
\usepackage{float} 
\usepackage{multirow}
\usepackage{footmisc}
\usepackage{rotating}

\graphicspath{ {./Figures/} }

\ifpdf
  \DeclareGraphicsExtensions{.eps,.pdf,.png,.jpg}
\else
  \DeclareGraphicsExtensions{.eps}
\fi

\newsiamremark{remark}{Remark}
\newsiamremark{hypothesis}{Hypothesis}
\crefname{hypothesis}{Hypothesis}{Hypotheses}
\newsiamthm{claim}{Claim}
\newtheorem{example}{Example}[section]

\def\P{{\mathbb P}}
\def\N{{\mathbb N}}
\def\R{{\mathbb R}}

\def\S{{\mathcal S}}
\def\E{{\mathcal E}}
\def\O{{\mathcal O}}

\def\bfp{{\bf p}}

\def\bfa{{\bf a}}
\def\bfb{{\bf b}}
\def\bfc{{\bf c}}
\def\bfd{{\bf d}}
\def\bfv{{\bf v}}
\def\bfx{{\bf x}}
\def\bfy{{\bf y}}
\def\bfz{{\bf z}}
\def\bfw{{\bf w}}
\def\bfu{{\bf u}}

\def\bfo{{\bf {\boldsymbol \nu}}}

\def\bfe{{\bf 1}}

\def\nhs{{\texttt{NM01}}}

\def\prox{{\rm{Prox}}_{\alpha\|(\cdot)_+\|_0}}
\def\pro{{\rm{Prox}}_{\tau\lambda\|(\cdot)_+\|_0}}

\def\ts{P-stationary point}

\def\beq{\begin{eqnarray}}
\def\eeq{\end{eqnarray}}
\headers{Smoothing Newton's Method for $0/1$-Loss Optimization}{ S. Zhou, L. Pan  N. Xiu and H. Qi}

\title{Quadratic Convergence of Smoothing Newton's Method for $0/1$ Loss  Optimization\thanks{Received by the editors April 1, 2021; accepted for publication (in revised form) September 5, 2021; published electronically December 13, 2021.\\
\href{https://doi.org/10.1137/21M1409445}{~~~~~https://doi.org/10.1137/21M1409445}
\funding{This work was funded by the the National Science Foundation of China (11971052, 11801325, 11771255) and Young Innovation Teams of Shandong Province (2019KJI013).}}
}

\author{Shenglong Zhou\thanks{Department of Electrical and Electronic Engineering, Imperial College London, London SW7 2AZ, United Kingdom. (\email{shenglong.zhou@imperial.ac.uk})}
\and 
Lili Pan\thanks{Department  of  Mathematics,  Shandong  University  of  Technology,  Zibo  255049, People's Republic of China.
(\email{panlili1979@163.com}).} 
\and 
Naihua Xiu\thanks{Department of Applied Mathematics, Beijing Jiaotong University, Beijing 10044, People's Republic of China.
(\email{nhxiu@bjtu.edu.cn}).}
\and 
  Hou-Duo Qi\thanks{School of Mathematics, University of Southampton, Southampton SO17 1BJ, United Kingdom.
  (\email{h.qi@soton.ac.uk}).}
}

\usepackage{amsopn}
\DeclareMathOperator{\diag}{diag}

\ifpdf
\hypersetup{
  pdftitle={Quadratic Convergence of Newton's Method for $0/1$-Loss Optimization},
  pdfauthor={S.L. Zhou, L.L. Pan, N.H. Xiu and H-D. Qi}
}
\fi

\nolinenumbers
\begin{document}

\maketitle

\begin{abstract}
It has been widely recognized that the $0/1$-loss function is one of the most natural choices for modelling classification errors, 
and it has a wide range of applications including support vector machines and $1$-bit compressed sensing. 
Due to the combinatorial nature of the $0/1$-loss function, methods based on convex relaxations or smoothing approximations have dominated the existing research and are often able to provide approximate solutions of good quality.
However, those methods are not optimizing the $0/1$-loss function directly and hence no optimality has been established for the original problem. 
This paper aims to study the optimality conditions of the $0/1$ function minimization, and for the first time to develop Newton's method that directly optimizes the $0/1$ function with a local quadratic convergence under reasonable conditions. Extensive numerical experiments demonstrate its superior performance as one would expect from Newton-type methods.
\end{abstract}
 
\begin{keywords}
 $0/1$-loss function, optimality conditions, Newton's method, locally quadratic convergence, superior numerical performance
\end{keywords}

\begin{AMS}
49M05, 90C26, 90C30, 65K05 
\end{AMS}
 
\section{Introduction}

This paper is concerned with the $0/1$-loss optimization:
\begin{equation}\label{ell-0-1}
\underset{\bfx\in\R^n}{\min}~ f(\bfx)+\lambda \|( A\bfx +\bfb )_+\|_0, 
\end{equation} 
where $f:\R^n\rightarrow\R$ is  twice continuously differentiable, 
$\lambda>0 $ is a penalty parameter and $A\in \R^{m\times n}, \bfb\in \R^m$. 
Moreover, ${\bfz}_+:=((z_1)_+, \ldots, (z_m)_+)^\top $ with $z_+:=\max\{z,{0}\}$ 
and $\|\bfz\|_0$ is the $\ell_0$ norm of $\bfz$, counting the number of its non-zero entries. 
Hence, $\|\bfz_+\|_0$ counts the number of positive entries of $\bfz$,
i.e., $\|\bfz_+\|_0=\sum_{i=1}^m \ell_{0/1}(z_i)$, where
\begin{eqnarray*}
\ell_{0/1}(z)=
\begin{cases}
1, &z>0,\\
0& z\leq 0.
\end{cases}
\end{eqnarray*} 
The function $\ell_{0/1}(\cdot)$ is known as the Heaviside step function (or the unit step function) in \cite{weisstein2002heaviside, evgeniou2000regularization} or simply the $0/1$-loss function in \cite{friedman1997bias, hastie2009elements, brooks2011support}.  It plays an active role in many applications including support vector machines (SVM) \cite{CV95}, the one-bit compressed sensing \cite{BB08}, the maximum rank correlation \cite{han1987non}, and the problem of area under curves \cite{ma2005regularized}.  However, optimization related to the $0/1$-loss function is NP-hard, see \cite{ben2003difficulty,feldman2012agnostic}.

A vast body of work has developed algorithms for optimization involving the $0/1$-loss function by making use of its continuous surrogates. A major concern on this part of research is that convergence analysis is often conducted on the surrogate problems rather than on the original ones that involve $0/1$-loss functions. On the other hand,  there also exists a large body of research that addresses the $0/1$-loss optimization directly by taking advantage of the intrinsic appealing feature of the loss function, which captures the discrete nature of the binary classification. We will review two classes of such methods below.

The first class consists of mixed integer programming (MIP), which has become a leading approach to directly optimizing the 0/1-loss function, see \cite{liittschwager1978integer, bajgier1982experimental} for some earlier work. It is straightforward to relate the $0/1$ loss to the misclassification minimization for discrimination problems \cite{rubin1997solving,brooks2010analysis}.  This approach is in general effective with a major issue of scalability for large-sized problems. Much progress has also been made in improving the scalability of MIP by employing various strategies of reducing the problem sizes.  Those include, for instance,   decomposition strategy \cite{rubin1997solving}, local search \cite{nguyen2013algorithms},  and convex hull cuts in a branch-and-bound framework \cite{brooks2011support}, to name a few.  Some recent work includes \cite{tang2014mixed,ustun2016supersparse}, where different integer programming reformulations for the 0/1-loss minimizations are built and then tackled via the modern commercial MIP solvers, such as Gurobi and CPLEX. Despite those progresses, the speed of computation is still a bottleneck for the MIP approach

The second class of methods comes from continuous optimization. Since the $0/1$-loss function is non-convex, non-differentiable and has zero gradients whenever differentiable,  coordinate descent directions are natural choices for decreasing the objective. There are a number of such methods including random coordinate descent algorithms \cite{LL2007}, greedy coordinate descent algorithms \cite{zhai2013direct}, and stochastic coordinate descent heuristic \cite{xie2019stochastic}. Convergence for this class of algorithms is often established in the probabilistic sense. Other types include a column generation approach \cite{carrizosa2010binarized} and an alternating direction method of multipliers \cite{wang2019support}, which is devoted to the $0/1$-loss regularized SVM problem. 

This paper aims to extend the classical Newton method to (\ref{ell-0-1}) and to prove its local quadratic convergence.  The investigation of Newton's method is motivated and supported by the following facts. 

\begin{itemize}
	\item[(i)]Newton's method has been recently developed by the authors in \cite{ZXQ21} for optimization problems with a sparse constraint $\| \bfx\|_0 \le s$.  Its performance is outstanding in comparison with a number of leading solvers that employ either hard- or soft-thresholding techniques. The essential difference of problem (\ref{ell-0-1}) from that in \cite{ZXQ21} is that our operator is a composite one that involves the operators $(\cdot)_+$,  $\| \cdot \|_0$, and the linear classification inequalities $A\bfx - \bfb \ge 0$.   Because of this, the framework developed in \cite{ZXQ21} cannot be used here. However, the success in \cite{ZXQ21} naturally leads us to investigate what form a Newton's method would take for (\ref{ell-0-1}) and whether it is computationally efficient. 
	
	\item[(ii)] In some important applications, the objective function $f(\bfx)$ is separable in the following form:
	\begin{eqnarray} \label{Block-f}
	  f(\bfx) = \sum_{i=1}^M f_i(\bfx_{(i)}),
	\end{eqnarray} 
	where each $\bfx_{(i)}$ ($i=1, \ldots, M$) is a subvector of $\bfx$ and not overlapping with each other.  Consequently, the Hessian of $f(\bfx)$ is block-diagonal. When each block is of small size, the inverse of Hessian (when exists) can be fast computed. In the particular applications of SVM and one-bit compressed sensing, the block-size is $1$ and the Hessian matrix is hence diagonal. One can imagine that Newton's method would be extremely efficient for such applications.  
    
    \item[(iii)] Although it is challenging to design a gradient-type method for the $0/1$-loss function due to its zero gradient (when exists),  we would like to emphasize that it is easy to compute the proximal operator of the $0/1$-loss function. Proximal operators have long been known to be closely related to optimality conditions in constrained optimization.  In particular, the proximal operator of the zero norm $\|\cdot\|_0$  characterizes a class of stationary points for sparse optimization and many hard- and soft-thresholding algorithms actually converge to such stationary points, see Beck and Eldar \cite{Beck13} for an excellent illustration. 
	
\end{itemize}

Our first step towards developing Newton's method is to establish a
stationary equation of the type:
\begin{equation} \label{F0}
F(\bfw; T):=\left[\begin{array}{rrl}
\nabla f(\bfx)  + A_{T}^\top \bfz_{T}  \\
A_{T} \bfx + \bfb_{T} \\
\bfz_{\overline T} 
\end{array}
\right] = 0 ,
\end{equation}
where $\bfw^\top := (\bfx; \bfz)$ with $\bfz := A \bfx + \bfb$,  
$T \subseteq \{1, \ldots, m\}$ is an index set
with $\overline{T}$ being its complementary set, and $A_T$ consists of the rows 
in $A$ indexed by $T$.    
This equation is characterized by the proximal operator of the $0/1$-loss function at a local minimum of (\ref{ell-0-1}). We call it the $P$-stationary  {(abbreviation for Proximal-stationary)} equation. See  \Cref{Stationary-Thm} and Equation (\ref{sta-eq-1}) for more details.  The index set $T$ depends on the optimal solution $\bfw^*$ and hence is unknown.

 The second step is to construct a scheme that defines $T_k$ at a given point $\bfw^k$ and approximate the true $T$. The Newton step is to solve the equation  
	$
	F(\bfw; T_k) = 0
	$
	for the next iterate $\bfw^{k+1}$. Such a computational scheme for 
	$T_k$ is described in (\ref{p-stationary-T}) and (\ref{T-k}). 
	However, the difficulty is that there is no guarantee that this scheme will be able to identify the correct $T$. In other words, we may be encountered with different $T_k$ each iteration no matter how close our iterate is to the optimal solution. This is where the convergence theory of classical Newton's method fails to go through.  
	Now we introduce a practically important technique of smoothing motivated by
	Chen et al. \cite{chen1998global}.
	Instead of solving the equation 
	$
	F(\bfw; T_k) = 0,
	$
	we try to solve its perturbed version:
	\begin{equation} \label{Fmu}
	F_{\mu_k} (\bfw; T_k) := F(\bfw; T_k) + 
	\left[
	\begin{array}{c}
	0 \\
	-\mu_k \bfz_{T_k} \\
	0
	\end{array} 
	\right] = 0 ,
	\end{equation}
	where the smoothing parameter $\mu_k>0$ will be properly chosen as in (\ref{update-mu}).
	Its Jacobian matrix has the following structure
	\[
	\nabla F_{\mu_k}(\bfw; T_k):=\left[\begin{array}{ccc}
	\nabla^2 f(\bfx)   & A_{T_k} ^\top &0\\
	A_{T_k} &-\mu_k I  &0 \\
	0 & 0 & I
	\end{array}
	\right]
	\]
	and it is nonsingular if and only if the matrix 
	$
	( \nabla^2 f(\bfx) +  A_{T_k} ^\top A_{T_k} /\mu_k)
	$
	is nonsingular (i.e., the Schur complement of $(-\mu_k I)$ in the top $2\times 2$ block is nonsingular). 
	We note that nonsingularity may still hold even if $f(\cdot)$ is not convex 
	provided that $A_{T_k}$ has full row-rank. 
	Due to its connection to \cite{chen1998global}, 
	we call our method a smoothing Newton method.
	We also like to point out another interesting connection. 
	When $\mu_k$ is fixed, our algorithmic framework is analogous to the primal-dual active-set algorithms extensively studied in \cite{hintermuller2002primal,ito2003semi,fan2014primal},
	whose main targets are quadratic objective functions.  

Our last step is to establish the bounds
\[
\| F(\bfw^{k+1}; T_{k+1}) \| = C_1 \| F(\bfw^k; T_k) \|^2, \qquad
\| \bfw^{k+1} - \bfw^* \| = C_2 \| \bfw^k - \bfw^*\|^2,
\]
where the constants $C_1$ and $C_2$ only depend on  the optimal solution $\bfw^*$.  Those bounds imply the quadratic convergence of our smoothing Newton method provided that the initial point is close to $\bfw^*$ and complete our theoretical investigation, see  \Cref{quadratic-convergence}. The efficiency of the Newton method is confirmed through extensive numerical experiments including $40$ SVM problems from real data ($16$ of them for $m \le n$ and $24$ for $m > n$)  and simulated $1$-bit compressing data against some existing solvers. As far as we know, this is the first Newton-type method for the $0/1$-loss optimization (\ref{ell-0-1}). 


This paper is organized as follows.  In the next section, we analyze the $0/1$-loss function, calculating its subdifferentials and the proximal operator.  In  \Cref{sec:opt},  we establish the first-order necessary and sufficient optimality conditions of the problem \eqref{ell-0-1} through the proximal operator of $0/1$-loss function, leading to the well-defined $P$-stationary points. In \Cref{sec:Newton}, we reformulate the $P$-stationary condition as a system of nonlinear equations and develop Newton's method with the promised quadratic convergence. In \Cref{sec:numerical}, we conduct extensive numerical experiments to demonstrate the outstanding performance of Newton's method against a few leading solvers for the problems of SVM and 1-bit compressed sensing.  We conclude the paper in \Cref{Sec-Conclusion}.

\section{Preliminaries} \label{Sec-Preliminaries}

We first list some notation that is frequently used throughout the paper. 
Given a subset $T\subseteq\N_m:=\{1,2,\ldots,m\}$, its cardinality and complementary set are $|T|$ and $\overline T$.   The neighbourhood of $\bfx\in\R^n$ with a radius $\delta>0$  is denoted by $N( \bfx,\delta)=\{\bfv\in\R^n: \|\bfv- \bfx \|< \delta\}$. 
Moreover, $\bfx_{T}$ (resp. $A_{T}$) represents the sub-vector (resp. sub-matrix) contains elements (resp. rows) of $\bfx$ indexed on $T$.   Particularly, $A_i$ is the $i$th row of $A$. 
We combine two vectors as $(\bfx;\bfy):=(\bfx^\top~\bfy^\top)^\top$. The $i$th   largest singular value of $H\in\R^{n\times n}$ is written $\sigma_{i}(H)$, namely $\sigma_{1}(H)\geq\sigma_{2}(H)\geq\cdots\geq\sigma_{n}(H).$
Particularly, we write $\|H\|:=\sigma_{1}(H)$ and $\sigma_{\min}(H):=\sigma_{n}(H).$ Finally, let $I$ be the identity matrix and $\bfe$ be the vector with all entries being ones. 

Next we describe the formula for computing the subdifferential of 
$\|\bfz_+\|_0$ and its proximal operator.
We note that $\|\cdot\|_0$ is lower semi-continuous (lsc)
and $\bfz_+$ is obviously continuous, 
the composition $\|\bfz_+\|_0$ is also lsc.
For a proper and lsc function $g: \R^m \mapsto \R$, its
subdifferential $\partial g(\cdot)$ is well defined as in \cite[Definition 8.3]{RW1998}. 
The following results are easy to prove after simple calculations.

\begin{lemma} \label{subdifferentials}

\begin{itemize}
	
 \item[(i)] We have
 \begin{eqnarray}\label{subfi}
  \partial\|{\bfy}_+\|_0
  =\left\{{\bfv}\in {\R}^m: v_i\left\{\begin{array}{ll}
 \geq0,& y_i=0,\\
 =0,& y_i\neq0,
 \end{array}
 \right. i\in{\N}_m\right\}. 
 \end{eqnarray}
 
 \item[(ii)] Let $g(\bfx):=\|(h(\bfx))_+\|_0$, where $h:\R^n\rightarrow\R^m$ is differentiable,  and $\Gamma:=\{i\in\N_m: h_i(  \bfx)=0\}$ for a given a point $ \bfx\in\R^n$.  If 
 \begin{eqnarray}
 \label{chain-rule-cond}
 \left.
 \begin{array}{r}
 \forall \bfz\in\R^{|\Gamma|},~~\bfz\geq 0,~~
 (\nabla h( \bfx))_{\Gamma}^\top \bfz = 0
 \end{array}\right.~~\Longrightarrow~~\bfz = 0,
 \end{eqnarray}
 then the subdifferential of $g(\bfx)$  at $  \bfx$ is  
 \begin{eqnarray}
 \label{chain-rule-sub} 
 \partial g(  \bfx)=\nabla h( \bfx)^\top \partial \|(h(  \bfx))_+\|_0.
 \end{eqnarray}

\end{itemize}
\end{lemma}

\begin{proof}
i)   It is straight to check that the regular subdifferential $\widehat{\partial}\|{\bfy}_+\|_0$ (see \cite[Definition 8.3]{RW1998})  of $\|{\bfy}_+\|_0$ at ${\bfy}$ takes the following form,
\begin{eqnarray}\label{subfi-0}
 \widehat{\partial}\|{\bfy}_+\|_0=\left\{{\bfv}\in {\R}^m: v_i\left\{\begin{array}{ll}
 \geq0,& y_i=0,\\
 =0,& y_i\neq0,
 \end{array}
 \right. i\in{\N}_m\right\}=:\Omega(\bfy). 
\end{eqnarray}
We next verify $ \partial\|{\bfy}_+\|_0= \Omega({\bfy}) $. By letting $\varphi(\cdot):=\|(\cdot)_+\|_0$, we have
  \begin{eqnarray}\label{sub-lim-reg}
  \arraycolsep=1.4pt\def\arraystretch{1.5}
  \begin{array}{rclll}
\partial\|{\bfy}_+\|_0&=&\limsup_{{\bfz}\stackrel{\varphi}{\rightarrow}{\bfy}}~ \widehat{\partial}\|{\bfz}_+\|_0 ~~~( \text{by  \cite[Equation 8(5)]{RW1998}})\\
&=&\limsup_{{\bfz}\stackrel{\varphi}{\rightarrow}{\bfy}}~\Omega({\bfz}) ~~~~~~( \text{by \eqref{subfi-0}})\\
&=&\{\bfv\in \R^m:\exists~\bfz^k\stackrel{\varphi}{\rightarrow} \bfy,~\bfv^k\rightarrow \bfv ~\text{with}~\bfv^k\in \Omega(\bfz^k) \}=:\Theta,
\end{array}
 \end{eqnarray}
where  $\bfz\stackrel{\varphi}{\rightarrow} \bfy$ represents $\bfz\rightarrow \bfy, \varphi(\bfz)\rightarrow \varphi(\bfy)$. Clearly, $\Omega({\bfy})\subseteq\Theta$. On the other hand, $\Theta \subseteq \Omega({\bfy})$  follows from that $\Omega(\bfz^k)\subseteq \Omega(\bfy)$ for any $\bfz^k\stackrel{\varphi}{\rightarrow} \bfy$ and $\Omega(\cdot)$ is closed.
 
 ii) Direct verifications yield the following chain of equations,
 \begin{eqnarray*}  \arraycolsep=1.4pt\def\arraystretch{1.5}
  \begin{array}{rclll}
 \partial^\infty \|\bfy_+\|_0&=& {\limsup}_{\sigma\downarrow0,~{\bfz}\stackrel{\varphi}{\rightarrow} {\bfy}} \sigma \widehat{\partial}\|\bfz_+\|_0&~~( \text{by \cite[Equation 8(5)]{RW1998}})\\ 
&=&
 {\limsup}_{{\bfz}\stackrel{\varphi}{\rightarrow} {\bfy}}~ \widehat{\partial}\|\bfz_+\|_0&~~( \text{by \eqref{subfi-0}})\\
 & =& \partial\|{\bfy}_+\|_0,&~~( \text{by \eqref{sub-lim-reg}})\\
 \end{array}
 \end{eqnarray*}
 where $ \partial^\infty \|{\bfy}_+\|_0$ is the horizon subdifferential of $\|{\bfy}_+\|_0$. Therefore, we derive that $\widehat{\partial}\varphi=\partial \varphi=\partial^\infty \varphi$.  One can easily prove that the horizon cone $\widehat{\partial}\varphi(\bfy)^\infty$ (see \cite[Definition 3.3]{RW1998}) of $ \widehat{\partial}\varphi(\bfy)$ satisfies $\widehat{\partial}\varphi(\bfy)^\infty = \partial^ \infty\varphi(\bfy)$. These conditions indicate that  the function $\varphi$ is regular by \cite[Corollary 8.11]{RW1998}, which together with \eqref{chain-rule-cond}, $g(\bfx)=\varphi(h(\bfx))$ and \cite[Theorem 10.6]{RW1998} derives \eqref{chain-rule-sub} immediately. 
\end{proof} 

The assumption in (\ref{chain-rule-cond}) can be regarded as a constraint qualification for the chain rule in (\ref{chain-rule-sub}) to hold. We finish this section with a formula to compute the proximal operator of  $\|(\cdot)_+\|_0$.  Let $\alpha>0$, the proximal operator of $\alpha \|(\cdot)_+\|_0$ at  {$\bfo$ is defined by 
\begin{eqnarray*} 
\prox(\bfo)=\underset{\bfy\in\R^m}{\rm argmin} ~\frac{1}{2}\|\bfy-\bfo\|^2+ \alpha \|\bfy_+\|_0.
\end{eqnarray*}
As shown in \cite[Lemma 2.2]{wang2019support}, the proximal operator admits a closed form as
\begin{eqnarray} \label{proximal-operator}
\Big[\prox(\bfo)\Big]_i=
\begin{cases}
0,&    \nu_i \in (0, \sqrt{2\alpha}),\\
0~ \text{or}~ \nu_i ,& \nu_i\in\{0,\sqrt{2\alpha}\},\\
\nu_i,&\nu_i\in(-\infty,0)\cup (\sqrt{2\alpha},\infty).
\end{cases} 
\end{eqnarray}}
\section{Optimality Conditions} \label{sec:opt}

In this section, we study the optimality conditions of \eqref{ell-0-1}  and characterize the conditions in terms of Proximal-stationarity  {(i.e., P-stationarity) using the proximal operator.}  Those results will lay down the foundation for Newton's method in the next section. For a given point $\bfx^*\in\R^n$, we denote 
\begin{eqnarray}\label{T-*}
  \Gamma_*&:=&\{i\in\N_m: A_i\bfx^*+b_i=0\}.
\end{eqnarray}
Our first result is to characterize a local minimizer of  \eqref{ell-0-1}.

\begin{lemma}\label{first-order-nec}
	The following relationships hold for the problem (\ref{ell-0-1}).
\begin{itemize}
	
\item[i)] A local minimizer $\bfx^*$   satisfies the following condition 
         if  $A_{\Gamma_*}$ is full row rank,
\begin{eqnarray} \label{lemma-1st-order-nec}
    -\nabla f(\bfx^*) \in A^\top \partial \|(A\bfx^*+\bfb)_+\|_0.
\end{eqnarray}
 
\item[ii)] A point $\bfx^*$ satisfying (\ref{lemma-1st-order-nec})  is a local minimizer if  the function $f$ is locally convex around $\bfx^*$.
\end{itemize}
\end{lemma}
\begin{proof}
i) It follows from \cite[Theorem 10.1]{RW1998} that a local minimizer of \eqref{ell-0-1} must satisfy $-\nabla f(\bfx^*)\in\lambda\partial g(\bfx^*)$, where $g(\bfx):=\|(A\bfx+\bfb)_+\|_0$. This together with \cref{chain-rule-sub}  and $\lambda \partial\|(\cdot)_+\|_0 = \partial \|(\cdot)_+\|_0$ by \eqref{subfi} derives the result immediately.  

ii)   Since the problem \eqref{ell-0-1} is equivalent to the following problem, 
\begin{eqnarray}\label{ell-0-1-equ-0}
&\underset{\bfx\in\R^n,\bfy\in\R^m}{\min}& f(\bfx)+\lambda \|\bfy_+\|_0,\\
  &{\rm s.t.}& A\bfx+\bfb-\bfy = 0,\nonumber
\end{eqnarray}
it suffices to show that  $(\bfx^*;\bfy^*)$  is a local minimizer of the   problem \eqref{ell-0-1-equ-0}, where $\bfx^*$ satisfies \eqref{lemma-1st-order-nec} and $\bfy^*=A \bfx^*+\bfb$, namely, there is a $\bfz^*$ such that
\begin{eqnarray}
\label{lemma-1st-order-nec-10}
\nabla f(\bfx^*) +A^\top\bfz^*=0,~~ 
A \bfx^*+\bfb-\bfy^*=0, ~~
\partial\|\bfy^*_+\|_0\ni \bfz^*. 
\end{eqnarray} 
It follows from $\bfy^*=A\bfx^*+\bfb$ and \eqref{T-*} that $\Gamma_*=\{i\in\N_m: y^*_i=0\}$. This together with $\partial\|\bfy^*_+\|_0 \ni \bfz^*$ and the expression of the $\partial\|\bfy^*_+\|_0$ in  \eqref{subfi} indicates 
\begin{eqnarray}\label{fact-11}
\bfy_{\Gamma_*}^*=0,~~\bfz_{\Gamma_*}^* \geq 0, ~~~~\bfy_{\overline \Gamma_*}^*\neq 0,~ ~\bfz_{\overline \Gamma_*}^* = 0. 
\end{eqnarray} 
 Define a radius $\delta:=\min\{\delta_1,\delta_2\}$, where
\begin{eqnarray}\label{delta-2}
~~~~~~\delta_1 := 
\begin{cases}
+\infty,&A_{\Gamma_*}^\top\bfz^* _{\Gamma_*}=0\\
\frac{\lambda}{ \|A_{\Gamma_*}^\top\bfz^* _{\Gamma_*}\|},& \text{otherwise}, 
\end{cases}~~~~
\delta_2 := 
\begin{cases}
+\infty,&\bfy^*\leq {\bf0},\\
\min_i\{y_i^*:y_i^*>0\},&\text{otherwise}, 
\end{cases}
\end{eqnarray}
and consider a local region of ${\bfw}^* := (\bfx^*;\bfy^*)$ by
\begin{eqnarray}\label{local-area-1}
~~N( {\bfw}^*,\delta)=\left\{ \left(
\bfx; 
  {\bfy}\right)\in\R^{n+m}   
:~A\bfx+\bfb - \bfy=0, ~
 \|{\bfw}- {\bfw}^*\|<\delta
\right\}.
\end{eqnarray}
Indeed, $N( {\bfw}^*,\delta)$ is a neighbourhood of ${\bfw}^*$ since $A\bfx^*+\bfb - \bfy^*=0$ from \eqref{lemma-1st-order-nec-10}. Next we show that, for any $\bfw\in N( {\bfw}^*,\delta)$,
\begin{eqnarray}\label{y-y*-0}
\|\bfy^*_+\|_0\leq \|\bfy_+\|_0.
\end{eqnarray}
Obviously, it is true if $\bfy^*\leq0$ as $\|\bfy^*_+\|_0=0$. For $\bfy^*\nleq0$,  to guarantee \eqref{y-y*-0}, it suffices to show that for any $i$, $y_i^*>0~\Longrightarrow~y_i>0.$ 
Suppose there is a $j\in\N_m$ such that $y_j^*>0$ but $y_j\leq0$.  This incurs the following contradiction 
\begin{eqnarray*}\delta_2 \geq \delta
& {>}&\|\bfw-\bfw^*\| \hspace{3.5cm} ({\rm by}~ \eqref{local-area-1})\\
&\geq& |y_j-y_j^*|=y_j^*-y_j  \geq y_j^*\geq\delta_2.\hspace{.5cm} ({\rm by}~ \eqref{delta-2}) \end{eqnarray*}
Again, for any $\bfw\in N( {\bfw}^*,\delta)$, we have $A\bfx+\bfb - \bfy=0$, which and \eqref{lemma-1st-order-nec-10} generate
 \begin{eqnarray}\label{y-y*-h-h*}
 \bfy-\bfy^*=A (\bfx- \bfx^*).
  \end{eqnarray}
Next, the convexity of $f$ gives rise to  
 \begin{eqnarray}\label{fx-fx*}
\hspace{15mm} f(\bfx)-f(\bfx^*)&\geq & \langle \nabla f(\bfx^*), \bfx - \bfx^*\rangle~~= -\langle  A^\top \bfz^* , \bfx - \bfx^*\rangle \hspace{1.6cm} ({\rm by}~ \eqref{lemma-1st-order-nec-10})\nonumber\\
 &=& -\langle  A_{\Gamma_*}^\top\bfz^* _{\Gamma_*},\bfx - \bfx^*\rangle= -\langle  A_{\Gamma_*}(\bfx - \bfx^*),\bfz^* _{\Gamma_*}\rangle=:\phi\hspace{0.3cm} ({\rm by}~ \eqref{fact-11})
  \end{eqnarray} 
 Now we make the conclusion by two cases.  
 If $\|\bfy^*_+\|_0= \|\bfy_+\|_0$, then   $\bfy_{\Gamma_*}\leq 0$ due to $\bfy^*_{\Gamma_*}=0$. This and \eqref{y-y*-h-h*} yield that
  \begin{eqnarray}\label{y-y*-h-h*-0}
0\geq  \bfy_{\Gamma_*}= \bfy_{\Gamma_*}-\bfy^*_{\Gamma_*}  = A_{\Gamma_*} (\bfx-\bfx^*),
  \end{eqnarray}
which together with $\bfz^* _{\Gamma_*}\geq0$ from \eqref{fact-11} indicates $\phi\geq0$, namely $f(\bfx)\geq f(\bfx^*)$. So $$f(\bfx) +\lambda \|\bfy_+\|_0 \geq f(\bfx^*) + \lambda\|\bfy_+^*\|_0.$$
If $\|\bfy^*_+\|_0 \neq \|\bfy_+\|_0$, we must have $ \|\bfy_+\|_0 \geq 1+\|\bfy^*_+\|_0$ by \eqref{y-y*-0}. If  $A_{\Gamma_*}^\top\bfz^* _{\Gamma_*}=0$, then $\phi=0>-\lambda$. Otherwise, it follows
  \begin{eqnarray*}
\phi&\geq&- \|A_{\Gamma_*}^\top\bfz^* _{\Gamma_*}\|\|\bfx - \bfx^*\| \geq - \|A_{\Gamma_*}^\top\bfz^* _{\Gamma_*}\|\|\bfw - \bfw^*\|\\
 & \geq&- \|A_{\Gamma_*}^\top\bfz^* _{\Gamma_*}\|\delta 
  \geq - \|A_{\Gamma_*}^\top\bfz^* _{\Gamma_*}\|\delta_1=-\lambda.\hspace{.5cm} ({\rm by}~ \eqref{delta-2})
  \end{eqnarray*}
Both lead to $\phi\geq-\lambda$, which results in
  \begin{eqnarray*}\arraycolsep=1.4pt\def\arraystretch{1.5}
  \begin{array}{rclll}
  f(\bfx) +\lambda \|\bfy_+\|_0 &\geq& f(\bfx^*) + \phi +\lambda \|\bfy_+\|_0 &\ ({\rm by}~ \eqref{fx-fx*} )\\
&\geq& f(\bfx^*) -\lambda + \lambda\|\bfy_+\|_0 &\ ({\rm by}~ \phi\geq-\lambda)\\
&\geq& f(\bfx^*) + \lambda\|\bfy_+^*\|_0.&\ ({\rm by}~ \|\bfy_+\|_0 \geq 1+\|\bfy^*_+\|_0 )
 \end{array} \end{eqnarray*}
Overall, the two cases show that $(\bfx^*;\bfy^*)$ is  a local minimizer to \eqref{ell-0-1-equ-0}. Namely, $\bfx^*$ is  a local minimizer to \eqref{ell-0-1}. 
\end{proof}


The characterization \eqref{lemma-1st-order-nec} is nice and it is in the
classic form of differential inclusion. 
However, the challenge is that it is difficult to extract second-order
information which is essential to Newton's method.
To this purpose, we continue to characterize it in terms of
P-stationarity.

\begin{definition}\label{def-p-sta} 
	A point $\bfx^*$ is called a P-stationary point of the problem (\ref{ell-0-1})  if there exist a constant $\tau > 0$ and
	a point $\bfz^*\in\R^m$  such that
\begin{eqnarray} \label{p-stationary}
 \begin{cases}
\hspace{21mm}\nabla f(\bfx^*)  + A^\top \bfz^* &=~ 0 \\
\pro (A\bfx^*+\bfb+\tau \bfz^*) &\ni~A\bfx^*+\bfb.
\end{cases}
\end{eqnarray}
\end{definition}

We also say a point $(\bfx^*;\bfz^*)$ is a P-stationary point of the problem (\ref{ell-0-1}) if it satisfies the conditions in \eqref{p-stationary}. 
For a point $\bfx^*$, we denote two constants by 
 \begin{eqnarray}
\label{tau-*} 
~~~~\tau_1 := 
\begin{cases}
+\infty, & \bfy^*\leq0,\\
{\min} \left\{\frac{(y^*_i)^2}{2\lambda}:y_i^*>0\right\},& {\rm otherwise},\\
\end{cases}~~
\tau_2:=\begin{cases}
+\infty, & \Gamma_*=\emptyset,\\
 \frac{2\lambda}{\max_i|p^*_i|^2},& {\rm otherwise},
\end{cases}
\end{eqnarray}
where $\bfy^*:=A\bfx^*+\bfb$ and $\bfp^*:=-(A_{\Gamma_*}A_{\Gamma_*}^{\top})^{-1}A_{\Gamma_*}\nabla f(  \bfx^*)$. Clearly, both $\tau_1>0$ and $\tau_2>0$. Based on these notation, we have the following main result of this section.
\begin{theorem} \label{Stationary-Thm}
   The following relationships hold for the problem (\ref{ell-0-1}).
\begin{itemize}
\item[i)]  A local minimizer $\bfx^*$ is a  P-stationary point for any $0<\tau<\tau_*:=\min\{\tau_1,\tau_2\}$ if  $A_{\Gamma_*}$ is full row rank. 
\item[ii)] A P-stationary point with $ \tau>0$  is a local minimizer if  the function $f$ is locally convex around $\bfx^*$. 
\item[iii)] A P-stationary point with $ \tau\geq\|A\|^2/c_f$  is   a global minimizer if  the function $f$ is strongly convex with a constant $c_f>0$.
\end{itemize}
\end{theorem}

\begin{proof} i) As  $\bfx^*$ is a local minimizer of \eqref{ell-0-1},  condition  \eqref{lemma-1st-order-nec} is valid by \cref{first-order-nec} if $A_{\Gamma_*}$ is  full row rank. In other words, there is a $\bfz^*$ such that
\begin{eqnarray}
\label{lemma-1st-order-nec-1}
\nabla f(\bfx^*) +A^\top\bfz^*=0,~~
A \bfx^*+\bfb-\bfy^*=0,~~
\bfz^*\in\partial\|\bfy^*_+\|_0. 
\end{eqnarray} 
Therefore, to show \eqref{p-stationary}, we only need to verify that,  for any $0<\tau<\tau^*$, $$\bfz^*\in \partial\|\bfy^*_+\|_0~ \Rightarrow~\bfy^*\in\P :=\pro (\bfy^*+\tau \bfz^*). $$ 
Recall the definition of $\Gamma_*$ in \eqref{T-*} and the second condition in \eqref{lemma-1st-order-nec-1}, we have $\bfy_{\Gamma_*}=(A \bfx^*+\bfb)_{\Gamma_*}=0$. Same reasoning also allows for obtaining \eqref{fact-11} due to $\bfz^*\in \partial\|\bfy^*_+\|_0$. 
As $A_{\Gamma_*}$ is full row rank, the first condition in \eqref{lemma-1st-order-nec-1} and  $A^\top\bfz^*  =A_{\Gamma_*}^\top\bfz^*_{\Gamma_*}$ derive  that
$$\bfz^*_{\Gamma_*}=\bfp^*.$$
Now, $0<\tau<\tau_*= \min\{\tau_1,\tau_2\}$ in \eqref{tau-*} results in 
\begin{eqnarray}\label{non-negative}
\arraycolsep=1.4pt\def\arraystretch{1.5}
\begin{array}{lllllllllll}
 y_i^*  & \geq & \underset{i:y^*_i>0}{\min} y^*_i &=&  \sqrt{2\tau_1\lambda}&\geq&\sqrt{2\tau_*\lambda} &>&\sqrt{2\tau \lambda}  &~~\text{if}~~& y_i^*>0,\\ 
   z_i^* & \leq &  {\underset{i}{\max}|p^*_i|} &= &\sqrt{2\lambda/\tau_2}&\leq&\sqrt{2\lambda/\tau_*}&< &\sqrt{2\lambda/\tau}& ~~ \text{if}~~ &z_i^*>0.
\end{array}
\end{eqnarray}
These and \eqref{fact-11}  yield the following condition,
\begin{eqnarray*} 
y_i^*\begin{cases}
=0, &   z_i^*=0 ~{\rm or}~ 0< z_i^* < \sqrt{2\lambda/\tau},\\
<0 ~{\rm or}~>\sqrt{2\tau\lambda },& z_i^*=0. 
\end{cases}
\end{eqnarray*}
  {It is easy to see that the above condition satisfies that 
\begin{eqnarray}\label{ugTT-1}
y_i^*  \in \P_i=
\begin{cases}
0, & y_i^*+\tau z_i^*\in(0, \sqrt{2\tau\lambda}),\\
0~{\rm or}~ y_i^*+\tau z_i^*, &y_i^*+\tau z_i^*\in\{0, \sqrt{2\tau\lambda}\},\\
y_i^*+\tau z_i^*,& y_i^*+\tau z_i^*\in(-\infty,0)\cup ( \sqrt{2\tau\lambda},\infty).
\end{cases}
\end{eqnarray}}

ii) Note that the second condition in \eqref{p-stationary} means $\bfy^*\in\P$, which by   \eqref{ugTT-1} implies $z_i^*=0$ if $y_i^*\neq 0$ and  $z_i^*\geq0$ if $y_i^*= 0$, resulting in $\bfz^*\in\partial\|\bfy^*_+\|_0$ by \eqref{subfi}. Consequently, we obtain \eqref{lemma-1st-order-nec-1} and \eqref{lemma-1st-order-nec}.
The claim follows from  \cref{first-order-nec} ii).

iii) Let $(\bfx^*;\bfz^*)$ be a  P-stationary point with $ \tau\geq\|A\|^2/c_f$. Then the second condition in \eqref{p-stationary} indicates that
\beq
&&({1}/{2\tau})\|A\bfx^*+\bfb-(A\bfx^*+\bfb+\tau \bfz^*)\|^2+\lambda\|(A\bfx^*+\bfb)_+\|_0 \nonumber \\
&\leq& ({1}/{2\tau})\|A\bfx+\bfb-(A\bfx^*+\bfb+\tau \bfz^*)\|^2+\lambda\|(A\bfx+\bfb)_+\|_0\nonumber
\eeq
for any $\bfx\in\R^n$, which after simplifying leads to 
\beq\label{Ax*b-Axb}
&& \lambda\|(A\bfx^*+\bfb)_+\|_0 - \|A(\bfx-\bfx^*)\|^2 /(2\tau) \nonumber\\
&\leq& \lambda\|(A\bfx+\bfb)_+\|_0 - \langle  \bfz^*, A(\bfx-\bfx^*) \rangle\nonumber\\
&=& \lambda\|(A\bfx+\bfb)_+\|_0 - \langle  A^\top \bfz^*, \bfx-\bfx^*    \rangle\nonumber\\
&=& \lambda\|(A\bfx+\bfb)_+\|_0 + \langle  \nabla f(\bfx^*), \bfx-\bfx^*  \rangle. \hspace{.5cm} ({\rm by}~  \eqref{p-stationary}) 
\eeq
The strong convexity of $f$ implies
\begin{eqnarray*} 
&&f(\bfx)+\lambda\|(A\bfx+\bfb)_+\|_0-f(\bfx^*)-\lambda\|(A\bfx^*+\bfb)_+\|_0\\
&\geq&({c_f}/{2})\|\bfx-\bfx^*\|^2+\langle  \nabla f(\bfx^*), \bfx-\bfx^*   \rangle+\lambda\|(A\bfx+\bfb)_+\|_0-\lambda\|(A\bfx^*+\bfb)_+\|_0\\
&\geq&({c_f}/{2})\|\bfx-\bfx^*\|^2-1/({2\tau})\|A(\bfx-\bfx^*)\|^2    \hspace{.25cm} ({\rm by}~  \eqref{Ax*b-Axb})\\
&\geq&({c_f/2-\|A\|^2/(2\tau}))\|\bfx-\bfx^*\|^2   \geq 0. \hspace{1cm} ({\rm by}~  \tau\geq\|A\|^2/c_f) 
\end{eqnarray*}
This shows the global optimality of $\bfx^*$ to the problem \eqref{ell-0-1}.
\end{proof}

\section{Smoothing Newton's Method}\label{sec:Newton}
 
The main purpose of this section is to formulate Newton's method and establish
its quadratic convergence. We first state two assumptions for this purpose.

\begin{assumption} \label{ass}
	 Suppose $f$ is twice continuously differentiable, $\nabla^2 f(\bfx^*)$ is {positive definite} and $A_{\Gamma_*}$ is full row rank, where $\Gamma_*$ is given by (\ref{T-*}).
 \end{assumption}
 
\begin{assumption} \label{ass2} 
Suppose $\nabla^2 f$ is locally Lipschitz continuous around $\bfx^*$ with a constant $L_*>0$, namely 
\begin{eqnarray*} \|\nabla^2 f(\bfx)-\nabla^2 f(\bfx')\|\leq L_* \|   \bfx -\bfx' \| \end{eqnarray*}
for any $\bfx$ and $\bfx'$ in the neighbourhood of $\bfx^*$.
 \end{assumption} 

\subsection{Stationary equations}
 For a point  $\bfw:=(\bfx; \bfz)$, we define the sets {
\begin{eqnarray}\label{p-stationary-T}
\arraycolsep=1.4pt\def\arraystretch{1.5}
 \begin{array}{rcll}
\S&:=& \left\{i\in\N_m:  A_i\bfx+b_i+\tau z_i   \in (0, \sqrt{2\tau\lambda}) \right\}, \\
 \E&:=& \left\{i\in\N_m: A_i\bfx+b_i +\tau z_i \in\{0,\sqrt{2\tau\lambda}\} \right\},\\
\O&:=& \left\{i\in\N_m:  A_i\bfx+b_i+ \tau z_i \in(-\infty,0)\cup ( \sqrt{2\tau\lambda},\infty) \right\},\\
\E^o&:=& \left\{i\in\N_m: A_i\bfx+b_i =0, \tau z_i   \in\{0,\sqrt{2\tau\lambda}\} \right\}, 
\end{array}
\end{eqnarray}}
for a given $\tau>0$. 
Obviously, $\E^o\subseteq \E.$ 
It is worth mentioning that all sets depend on $\bfw$. 
For simplicity, we drop their dependence whenever 
there is no confusion to be caused. 
Same rules are also applied into $\S_*,\E_*,\O_*$ and $\E_*^o$ for  $\bfw^*:=(\bfx^*; \bfz^*)$. 
A key step towards the Newton method is the construction of
the following system of equations.
For a given subset $\Gamma\subseteq\N_m$ and a scalar $\mu\geq0$, 
it follows the definitions in (\ref{F0}) and (\ref{Fmu}) that
\begin{eqnarray}
\label{sta-eq-1}
\hspace{8mm} F(\bfw; \Gamma) =\left[\begin{array}{rrl}
\nabla f(\bfx)  + A_{\Gamma}^\top \bfz_{\Gamma}  \\
A_{\Gamma} \bfx + \bfb_{\Gamma} \\
\bfz_{\overline \Gamma} 
\end{array}
\right] = 0, \quad 
\nabla F_{\mu}(\bfw;\Gamma) =\left[\begin{array}{ccc}
\nabla^2 f(\bfx)   & A_{\Gamma} ^\top &0\\
A_{\Gamma} &-\mu I&0 \\
0&0&I
\end{array}
\right].
\end{eqnarray} 
We note that the matrix $\nabla F_{\mu}$ is a slight perturbation of
the Jacobian matrix of $F$ and 
$\nabla F(\bfw;\Gamma)=\nabla F_{0}(\bfw;\Gamma)$,
The following result relates a P-stationary point to a system of equations. 

\begin{theorem}\label{cor-p-staionary}
A point $\bfw^*=(\bfx^*;\bfz^*)$ is a  P-stationary point with $\tau>0$ of the problem (\ref{ell-0-1}) if and only if   $F(\bfw^*; \Gamma_*)=0$ and $ \Gamma_*=(\S_* \cup \E^o_*)$, where $\Gamma_*$ is defined by (\ref{T-*}).
The Jacobian  $\nabla F(\bfw^*;\Gamma_*)$ is   non-singular if \cref{ass} holds.
\end{theorem}

\begin{proof} 
The second claim is obvious due to Assumption \ref{ass}.
 We only need to prove the first claim. 
We start with the sufficiency.
 The definitions  in \eqref{p-stationary-T} show the following relationships,
\beq\label{TSE}
\Gamma_*= \S_* \cup \E^o_*,~~~~\overline{\Gamma}_*=\O_*\cup(\E_* \setminus \E^o_*).
\eeq
We recall $\bfy^*=A\bfx^*+\bfb$ and $\P:=\text{Prox}_{\tau\lambda\|({\cdot})_{+}\|_{0}}( \bfy^*+\tau \bfz^*)$.
It follows that 
 \allowdisplaybreaks \begin{eqnarray*} 
 \P_i 
=
\begin{cases}
0, & i\in \S,\\
0~{\rm or}~y^*_i+\tau z^*_i, &i\in\E_*,\\
 y_i+\tau z^*_i,& i\in \O_*,
\end{cases} 
=
\begin{cases}
0, & i\in \S_*\subseteq \Gamma_*,\\
0~{\rm or}~\tau z^*_i, &i\in \E^o_* \subseteq \Gamma_*,\\
0~{\rm or}~ y^*_i, &i\in(\E_*\setminus \E^o_*) \subseteq \overline \Gamma_*,\\
y^*_i,& i\in \O_*\subseteq \overline \Gamma_*,
\end{cases} 
\end{eqnarray*}
where the first equation is by \eqref{proximal-operator} and second one is by \eqref{sta-eq-1} and \eqref{TSE}, which indicates $y^*_i \in \P_i$ due to $ \bfy^*_{\Gamma_*}=0$ from \eqref{sta-eq-1}. Moreover,   
the first and third equations in \eqref{sta-eq-1} suffice to $\nabla f(\bfx^*)  + A^\top \bfz^*  =  0$, showing \eqref{p-stationary}. Namely,  $ \bfw^*$ is a  P-stationary point. 

{Necessity.}  Let $ \bfw^*$ be a  P-stationary point satisfying \eqref{p-stationary} and $T_*:= \S_* \cup \E^o_*$.  Then $\overline{T}_*=\O_*\cup(\E_* \setminus \E^o_*)$. It follows from $y^*_i\in \P_i$ and \eqref{proximal-operator} that
\begin{eqnarray*} 
y^*_i\in 
\begin{cases}
0, & i\in \S_*,\\
0~{\rm or}~ y^*_i+\tau z^*_i, &i\in\E_*,\\
y^*_i+\tau z^*_i,& i\in \O_*,
\end{cases} =
\begin{cases}
0, & i\in \S_*,\\
0~{\rm or}~\tau z^*_i, &i\in\E^o_*,\\
0~{\rm or}~ y^*_i+\tau z^*_i, &i\in\E_*\setminus \E^o_*,\\
y^*_i+\tau z^*_i,& i\in \O_*,
\end{cases} 
\end{eqnarray*}
where the equality is by the definition  of $\E^o_*$, which together with $ y^*_i= 0 , i\in \E^o_*$ and $y^*_i\neq 0, i\in \E_*\setminus \E^o_*$ suffices to
\begin{eqnarray*} 
\begin{cases}
y^*_i=0, & i\in \S_*\cup\E^o_*=T_*,\\
y^*_i = y^*_i+\tau z^*_i, &i\in\E_*\setminus \E^o_*,\\
y^*_i= y^*_i+\tau z^*_i,& i\in \O_*,
\end{cases} ~~\Longleftrightarrow~~
\begin{cases}
y^*_i &=0, ~ i\in T_*,\\
 z^*_i&=0, ~i\in\overline{T}_*.
\end{cases} 
\end{eqnarray*}
This gives rise to the last two conditions $\bfy^*_{T_*}=0,~
\bfz^*_{\overline T_*}=0$  in \eqref{sta-eq-1}. 
Furthermore, the first condition in \eqref{p-stationary} and $\bfz^*_{\overline{T}_*}=0$ derive the first condition of \eqref{sta-eq-1}.
Overall, we have $F(\bfw^*; T_*)=0$. 
Now we show $T_*=\Gamma_*$. By  (\ref{T-*}) that $\Gamma_*=\{i\in\N_m: y^*_i=0\}$, it follows $T_*\subseteq \Gamma_*$. Suppose, there is a $j\in \Gamma_*$ but $j\notin T_*$, then
we have $y_j^*=z_j^*=0$ and thus $j\in \E^o_* \subseteq T_*$  by \eqref{p-stationary-T}, a contradiction. Therefore, $T_*=\Gamma_*$, finishing the proof. \end{proof}

\begin{remark}\label{rem-tau-lam}
It is interesting to note that \cref{cor-p-staionary}  suggests a threshold value for $\lambda$ to exclude the zero solutions when $b_i\neq0,\forall i\in \N_m.$  Suppose $\bfx^*=0$ is a P-stationary point. The second equation $A_{\Gamma_*} \bfx + \bfb_{\Gamma_*} =0$ in  $F(\bfw^*;\Gamma_*)=0$ indicates $\Gamma_*=\emptyset$ and thus $\bfz^*=0$.  These and \eqref{p-stationary-T} give rise to  $b_i\in(-\infty,0)\cup [ \sqrt{2\tau\lambda},\infty),\forall i\in \N_m.$ In real applications (e.g., SVM and 1-bit CS), there is at least one $i\in \N_m$ such that $b_i >0$, which results in $\lambda\leq  \min_i\{b_i^2: b_i>0\}/({2\tau})$.  Hence, to exclude the zero solutions for some real applications, we choose 
\begin{eqnarray}\label{lambda-lower-bound}
\lambda >  {\min_i\{b_i^2: b_i>0\}}/{(2\tau)}.
\end{eqnarray}
\end{remark}

\subsection{Algorithmic design}  

\cref{cor-p-staionary} lays the foundation for developing Newton's method,
which is to solve the stationarity equation in \eqref{sta-eq-1}.
Let $\bfw^k:=(\bfx^k; \bfz^k)$ be the current iterate.
We define $\S_k$ and $\E_k^o$ by \eqref{p-stationary-T} with $\bfw$ being replaced by $\bfw^k$ and let  
  \beq \label{T-k}
      T_k := \S_k\cup \E_k^o.  
 \eeq
Let $\bfd^k=(\bfu^k;\bfv^k)$ with $\bfu^k\in\R^n$ and $\bfv^k\in\R^{m}$. For 
such a defined $T_k$, a Newton direction $\bfd^k$ for the equation \eqref{sta-eq-1} solves the following linear equations:
 \[ 
 \nabla F(\bfw^k;T_k)~\bfd = - F(\bfw^k;T_k). 
 \] 
To improve the nonsingularity of the Jacobian matrix $\nabla F(\bfw^k;T_k)$,
we replace it with $\nabla F_{\mu_k}(\bfw^k;T_k)$. 
That is, at $\bfw^k$, we solve the equation:
 \beq \label{newton-dir-0}
 \nabla F_{\mu_k}(\bfw^k;T_k)~\bfd = - F(\bfw^k;T_k), 
 \eeq
where $\nabla F_{\mu_k}(\bfw^k;T_k)$ is  defined in \eqref{sta-eq-1}.
The Newton direction $\bfd^k$ satisfies 
  \beq \label{newton-dir} 
 \left[  
     \begin{array}{ccc}
     {\nabla^2  f(\bfx^k)}& A_{T_k}^\top&0\\
   A_{T_k}&-\mu_k I&0\\
    0&0&I
     \end{array}
     \right]  
     \left[  
     \begin{array}{l}
    \bfu^k\\
    \bfv^k_{T_k}\\
    \bfv^k_{\overline T_k}
     \end{array}
     \right]
     =-\left[  
     \begin{array}{r}
    \nabla  f(\bfx^k)+A_{T_k}^\top\bfz^k_{T_k}\\
   A_{T_k}\bfx^k+\bfb_{T_k} \\
    \bfz^k_{\overline T_k}
     \end{array}
     \right].
 \eeq
Here, the rule to update $\mu_k$ is as follows:
   \beq \label{update-mu} 
\mu_{k}=
\min\{\alpha \mu_{k-1}, \rho\|F(\bfw^k;T_k)\|\},
 \eeq
 where $\alpha\in(0,1)$ and $\rho>0$. Now we summarize   the proposed method in  \cref{Alg-NM01}. 
\begin{algorithm}[H]
	\caption{\nhs: Newton's method for $0/1$-loss optimization}
	\begin{algorithmic}[1] \label{Alg-NM01}
		\STATE Initialize $\bfw^0=(\bfx^0;\bfz^0)$ and $\mu_{-1}>0$. Set the parameter $\tau,\lambda,\rho>0, \alpha\in(0,1)$.\\
		 Compute $T_0$ by \eqref{T-k} and set $k :=0$.
\IF{ $\|F(\bfw^k;T_k)\|>0$}		 
         \STATE  Update $\mu_{k}$ by \eqref{update-mu}.	
		\STATE  Update $\bfd^k$ by solving \eqref{newton-dir}.
		\STATE  Update $\bfw^{k+1}~=~\bfw^k+\bfd^k$.
		\STATE  Update $T_{k+1}$ by \eqref{T-k} and  set $k := k+1$.   		
\ENDIF		
\RETURN $\bfw^{k}$.
	\end{algorithmic}
\end{algorithm}

\begin{remark} 
In general,  the computational complexity for solving the equation \eqref{newton-dir} is approximately $O(n^2\max\{n,|T_k|\})$.  This is fine for small-sized problems.  When $n$ is large, the computational cost is too high and existing first-order algorithms would be faster. Fortunately, for many real applications, such as SVM and 1-bit CS, their functions are separable and have block structures such as in \eqref{Block-f}.  This implies that the Hessian matrix $  {\nabla^2  f(\bfx^k)}$ is of diagonal blocks and is invertible.  The worst-case computational complexity can be reduced to $O(|T_k|^2\max\{n,|T_k|\})$.   For SVM or 1-bit CS problems, $T_k$ coincides with the indices of incorrectly classified samples that take a relatively small portion of the total samples. Hence, $|T_k|$  can be on a small scale and computation of the Newton direction can be very cheap. 
\end{remark}

\subsection{Quadratic convergence}

Let us first explain why it is a challenging task to establish the quadratic convergence of the proposed Newton method. Suppose $\bfw^*$ satisfies the stationarity equation $F(\bfw^*; \Gamma_*) =0$ (see \cref{cor-p-staionary}). If we know $\Gamma_*$ beforehand, then by fixing $T_k=\Gamma_*$,  our proposed method reduces to the standard Newton's method that solves equations with smooth functions.  The quadratic convergence follows under \cref{ass} and \cref{ass2}.  However, the difficulty we are facing is that the set $T_k$  may change from iteration to iteration.  A different $T_k$ leads to a different system  of equations $F(\bfw;T_k)=0$. Hence, in each step, the algorithm finds a Newton direction for a different system of equations instead of a fixed system.  This is where the standard proof for quadratic convergence breaks down. As we will see below, it takes a great deal of effort in establishing quadratic convergence.

The first technical result is about extending the stationarity equation to some indices that are given in a neighborhood of $\bfw^*$. In the proof, we recall $\bfy = A\bfx+\bfb$ and $\bfy^* = A\bfx^*+\bfb$. 

\begin{lemma}\label{lemma-neighbour-1}
	 Let $\bfw^*$ be a \ts\  with $0<\tau<\tau_*:=\min\{\tau_1,\tau_2\}$ of the problem (\ref{ell-0-1}),  $\tau_1,\tau_2$ and $\Gamma_*$  be given by (\ref{tau-*}) and (\ref{T-*}).  Then there is a $\delta_1^*>0$ such that, for any $\bfw\in N(\bfw^*,\delta_1^*)$ with its associated indices $\S$ and $\E^o$, it holds 
 \beq\label{gw*-0} 
 F(\bfw^*;T)=0 ~~~~{\rm and}~~~T:=(\S\cup\E^o)\subseteq \Gamma_*.
 \eeq
\end{lemma}

\begin{proof} i) \cref{cor-p-staionary} states that the \ts\ $\bfw^*$ of (\ref{ell-0-1}) satisfies
\begin{eqnarray}
\label{sta-eq-1-*}
\nabla f(\bfx^*)  + A_{\Gamma_*} \bfz^*_{\Gamma_*}=0,~~
 \bfy^*_{\Gamma_*}=0, ~~
\bfz^*_{\overline \Gamma_*}  =0 
\end{eqnarray} 
 for $0<\tau<\tau_*$, where $\Gamma_*=\S_*\cup\E^o_*$. Note that $\E_*\setminus\E^o_*\subseteq \overline{\Gamma}_*$ which by \eqref{sta-eq-1-*} leads to
\begin{eqnarray} \label{sta-eq-1-*-z}
   \bfz^*_{\E_*\setminus\E^o_*}=0.
\end{eqnarray} 
Using the same reasoning for proving \eqref{non-negative},  
we can prove for $ 0<\tau<\tau_*= \min\{\tau_1,\tau_2\}$ in \eqref{tau-*} that
\begin{eqnarray} \label{yz-tau-lambda} 
 y_i^*  >\sqrt{2\tau\lambda}  ~~{\rm if}~~y_i^*> 0,~~~~~~ \tau z^*_i <   \sqrt{2\tau \lambda   } ~~{\rm if}~~z^*_i> 0.
\end{eqnarray}
Therefore,  we have the following facts {
\allowdisplaybreaks
\begin{eqnarray*}
\arraycolsep=1.4pt\def\arraystretch{1.5}
\begin{array}{rcll}
\E^o_*&=& \{i\in\N_m: y_i^*=0, \tau z_i^* \in \{0, \sqrt{2\tau\lambda}\} \}&~~(\text{by} ~\eqref{p-stationary-T})\\
&=& \{i\in\N_m: y_i^*=0, z_i^*=0 \}, &~~(\text{by}~\eqref{yz-tau-lambda})\\
\E_*\setminus\E^o_*&=& \{i\in\N_m: y_i^*\neq0, y_i^*+\tau z_i^* = \sqrt{2\tau\lambda } \}&~~(\text{by} ~\eqref{p-stationary-T})\\
 &=& \{i\in\N_m: y_i^* = \sqrt{2\tau\lambda} \}&~~(\text{by}~\eqref{sta-eq-1-*-z})\\
&=&\emptyset, &~~(\text{by}~\eqref{yz-tau-lambda})
 \end{array}\end{eqnarray*}}
These facts lead to
\begin{eqnarray}\label{E*0}
\E_*=\E^o_*\cup(\E_*\setminus\E^o_*)=\E^o_*
=\Big\{i\in\N_m: y_i^*= z_i^*=0\Big\},
\end{eqnarray}
which yields the following relations
\begin{eqnarray}\label{TSETO}
\Gamma_*=\S_*\cup \E_*^o=\S_*\cup \E_*,~~~~\overline \Gamma_*= \O_*. 
\end{eqnarray}   
 {For a sufficiently small $\delta_1^*$, any $\bfw\in N(\bfw^*,\delta_1^*)$ satisfies,
\allowdisplaybreaks\begin{eqnarray}\label{hz-h*z*}
&&  {| y_i +\tau z_i - y_i^* -\tau z^*_i |\leq c\delta_1^*},~~\forall~i\in\N_m,   
\end{eqnarray} 
 {where $c>0$ is a constant relied on $A$ and $\tau$.} The definitions of $\S$ and $\S_*$  in \eqref{p-stationary-T} mean that, for any $i\in\S$ or $i\in\S_*$,
\begin{eqnarray} \label{facts-S}
\arraycolsep=1.4pt\def\arraystretch{1.5}
\begin{array}{rcll}
 y_i+\tau z_i\in (0, \sqrt{2\tau\lambda})~~ &\Longleftrightarrow&~~  |y_i+\tau z_i-\sqrt{\tau\lambda/2} |< \sqrt{\tau\lambda/2},\\
y_i^*+\tau z_i^*\in (0, \sqrt{2\tau\lambda})~~ &\Longleftrightarrow&~~  |y_i^*+\tau z_i^*-\sqrt{\tau\lambda/2} |< \sqrt{\tau\lambda/2}.
\end{array}
\end{eqnarray} 
Using this fact, if $\S_* \nsubseteq \S$, then there is an $i\in \S_*$ but $i\notin\S$ such that
\allowdisplaybreaks
\begin{eqnarray*}  
\arraycolsep=1.4pt\def\arraystretch{1.5}
\begin{array}{rcll}
|y_i +\tau z_i - y_i^* -\tau z^*_i|
&=&|y_i+\tau z_i-\sqrt{\tau\lambda/2}  - y_i^*-\tau z^*_i+\sqrt{\tau\lambda/2}|\nonumber\\
&\geq&   \sqrt{\tau\lambda/2} -|y_i^*+\tau z^*_i-\sqrt{\tau\lambda/2}| ~~~~ (\text{by}~i\notin\S~\text{and}~\eqref{facts-S})\\
  &\geq&   \sqrt{\tau\lambda/2} -\max_{i\in\S_*} |y_i^*+\tau z^*_i-\sqrt{\tau\lambda/2}| =:\delta_s\\
&>&0, ~~~~(\text{by}~\eqref{facts-S})
\end{array}
\end{eqnarray*}  {Since $c\delta_1^*$} can be smaller than $\delta_s$, the above fact contradicts with \eqref{hz-h*z*}.  }Hence, it holds $\S_* \subseteq \S$. Similar reasoning also derives $\O_* \subseteq \O$. These allow us to obtain $$\E = \N_m \setminus (\S\cup\O)\subseteq \N_m \setminus (\S_*\cup\O_*)= \E_*.$$
Overall, for any $\bfw\in N(\bfw^*,\delta_1^*)$, it holds
\begin{eqnarray}\label{SSOOEE}
\S_* \subseteq \S,~~ \O_* \subseteq \O,~~ \E\subseteq \E_*.
\end{eqnarray} 
The above relations enable us to claim that 
\begin{eqnarray}\label{SS*E*}
(\S\setminus \S_*) \subseteq \E_*,~~~~(\O\setminus \O_*) \subseteq \E_*.
\end{eqnarray}
Now, we can show
\allowdisplaybreaks
\begin{eqnarray*} 
\arraycolsep=1.4pt\def\arraystretch{1.5}
\begin{array}{rcll}
\bfy^*_{ \S_*}&=&0,&~~(\text{by}~\S_*\subseteq \Gamma_*~\text{and}~\eqref{sta-eq-1-*})\\
\bfy^*_{\S\setminus \S_*} &=&0, &~~(\text{by}~ \S\setminus \S_*\subseteq\E_*~\text{from}~\eqref{SS*E*}~\text{and}~\eqref{E*0})\\
\bfy^*_{\E^o} &=&0,&~~(\text{by}~ \E^o\subseteq\E\subseteq\E_*~\text{from}~\eqref{SSOOEE}~\text{and}~\eqref{E*0})\\
\bfz^*_{ \O_*}  &=&0, &~~(\text{by}~\O_*\subseteq \overline \Gamma_*~\text{and}~\eqref{sta-eq-1-*})\\
\bfz^*_{\O\setminus \O_*} &=&0, &~~(\text{by}~ \O\setminus \O_*\subseteq\E_*~\text{from}~\eqref{SS*E*}~\text{and}~\eqref{E*0})\\
\bfz^*_{\E\setminus\E^o} &=&0.&~~(\text{by}~ \E\setminus\E^o\subseteq\E\subseteq\E_*~\text{from}~\eqref{SSOOEE}~\text{and}~\eqref{E*0})
\end{array}
\end{eqnarray*}
These conditions combining with
\begin{eqnarray*}T&=&\S\cup\E^o=\S_*\cup(\S\setminus \S_*)\cup\E^o,\\
 \overline T&=&\O\cup(\E\setminus\E^o)=\O_*\cup(\O\setminus \O_*)\cup(\E\setminus\E^o),
 \end{eqnarray*}
imply $\bfy^*_{ T} =0$ and $\bfz^*_{\overline T}=0$. 
As a consequence of this and $\bfz^*_{\overline \Gamma_*}=0$ from \eqref{sta-eq-1-*}, 
\begin{eqnarray*}
  0= f(\bfx^*)+ A_{\Gamma_*} ^\top\bfz^*_{\Gamma_*}=f(\bfx^*)+ A^\top\bfz^*= f(\bfx^*)+ A_{T} ^\top\bfz^*_{T}.
 \end{eqnarray*}
Overall, we verify $F(\bfw^*;T)=0$,  as desired. 
Finally, we observe that
 \beq 
\arraycolsep=1.4pt\def\arraystretch{1.5}
 \begin{array}{rcll}
T=\S\cup\E^o&=&\S_*\cup(\S\setminus \S_*)\cup\E^o\nonumber\\
&\subseteq&   \S_*\cup\E_* \cup\E^o&(\text{by}~ \S\setminus \S_*\subseteq\E_*~\text{from}~\eqref{SS*E*})\nonumber\\
&\subseteq&   \S_*\cup\E_* \cup \E_*  &(\text{by}~\E^o\subseteq\E\subseteq\E_*~\text{from}~\eqref{SSOOEE})\nonumber\\
&=&   \Gamma_*.  &(\text{by}~\eqref{TSETO})\nonumber
\end{array}
 \eeq
The whole proof is completed.   \end{proof}

The second technical result is about the uniform nonsingularity of
the perturbed Jacobian matrix $\nabla F_{\mu}(\bfw;T)$ over a 
neighborhood of $\bfw^*$.

\begin{lemma}\label{lemma-neighbour-2} Let $\bfw^*$ be a \ts\  with $0<\tau<\tau_*$ of the problem (\ref{ell-0-1}) and $\tau_*$ be given by (\ref{tau-*}).   Assume  \cref{ass} and \cref{ass2}. It holds 
 \beq \label{bd-jacobian-F-w-T}
  && C_*  \geq \|\nabla F_{\mu}(\bfw;T)\|\geq\sigma_{\min}( \nabla F_{\mu}(\bfw;T)) \geq   c_*>0, 
\eeq
for any $\bfw\in N(\bfw^*,\delta_2^*)$ and any $0\leq\mu\leq c_*/2$, where $T=\S\cup\E^o$ and
 \beq\label{radius-2}
C_*&:=&2\max\{1, \|H(\Gamma_*)\|\} ,~~~~ c_* := 0.5{\min\{1,\min_{\Gamma\subseteq\Gamma_*}\sigma_{\min}(H(\Gamma))\}}, \\
\label{H*}\delta_2^* &:=& \min\left\{\delta_1^*,   \frac{c_*}{ 2L_*}     \right\},~~~~H(\Gamma):=\left[  
     \begin{array}{ccc}
    \nabla^2 f(\bfx^*)& A^\top_{\Gamma}\\
     A_{\Gamma}&0
     \end{array}
     \right]. 
\end{eqnarray}
\end{lemma}

\begin{proof}  
Since  $\nabla^2 f(\bfx^*)$ is  {positive definite} and $A_{\Gamma_*}$ is full row rank by \cref{ass}, $A_\Gamma$ is full row rank for any $ \Gamma \subseteq   \Gamma_*$ and thus $H(\Gamma)$ is non-singular. 
We have $\sigma_{\min}(H(\Gamma)) > 0$ for any 
$ \Gamma \subseteq   \Gamma_*$ 
and  $c_*>0$. 
Now we build the bounds of $\nabla F_{\mu}(\bfw;T)$. 
For any given two matrices $D'$ and $D$, we have the first fact
\begin{eqnarray} \label{lower-bd-HH}
 \|D'-D\|
&\geq& \max_{i}|\sigma_i(D')-\sigma_i(D )|   \geq |\sigma_{i_0}( D')-\sigma_{i_0}(D)|\nonumber\\
&\geq& \sigma_{i_0}(D')-\sigma_{\min}(D) \geq \sigma_{\min}(D')-\sigma_{\min}(D), 
\end{eqnarray}
where the first inequality is from \cite[Reminder (2), on Page 76]{lutkepohl1996handbook} and  $i_0$ satisfies that $\sigma_{i_0}(D)=\sigma_{\min}(D)$.  Recall that
 \begin{eqnarray*}
 \label{FH*-FG}H(\Gamma)\overset{\eqref{H*}}{=}\left[  
     \begin{array}{ccc}
    \nabla^2 f(\bfx^*)& A^\top_{\Gamma }\\
     A_{\Gamma }&0
     \end{array}
     \right],~~~~ \nabla F(\bfw^*;T) =\left[  
     \begin{array}{cc}
    H(T)&0\\
     0&I  
     \end{array}
     \right].
    \end{eqnarray*}
For any $\bfw\in N(\bfw^*,\delta_2^*)$ and $\delta_2^*\leq\delta_1^*$,   \cref{lemma-neighbour-1} contributes to $T \subseteq   \Gamma_*$. So $H(T)$ is a submatrix of $H(\Gamma_*)$. Hence, 
\beq 
\|H(T)\| \leq \|H(\Gamma_*)\|,~~ \sigma_{\min}(H(T)) \geq \min_{\Gamma \subseteq \Gamma_*} \sigma_{\min} (H(\Gamma)),\nonumber
\eeq
where the latter is by $T \subseteq   \Gamma_*$, which gives us the second fact 
\beq\label{JJogG0}
\sigma_{\min}(\nabla F(\bfw^*;T))&=&\min\{1,\sigma_{\min}(H(T))\}\\
& \geq&\min\{1, \min_{\Gamma \subseteq \Gamma_*} \sigma_{\min} (H(\Gamma))\}=2c_*, \nonumber\\
\label{JJogG1}
\|\nabla F(\bfw^*;T)\| &=& \max\{1, \|H(T)\|\} \leq \max\{1, \|H(\Gamma_*)\|\} =C_*/2. 
\eeq
The locally Lipschitz continuity of $\nabla^2 f$  around $\bfx^*$ with $L_*$  yields the third fact,
\begin{eqnarray}
\label{HH*2}
\|\nabla F (\bfw^*;T)-\nabla F (\bfw;T)\|&=&\|\nabla^2  f (\bfx)-\nabla^2  f (\bfx^*)\|  \leq
   L_* \| \bfx -\bfx^* \| \nonumber\\
   &\leq& L_* \| \bfw -\bfw^* \| 
 \leq   L_* \delta_2^*
\leq c_*/2.   \hspace{.5cm}(\text{by}~\eqref{H*})
\end{eqnarray} 
Now these three facts allow us to derive
\allowdisplaybreaks 
\begin{eqnarray}\label{lower-bd-FF}
\arraycolsep=1.4pt\def\arraystretch{1.5}
\begin{array}{rclr}
&& \sigma_{\min}(\nabla F_{\mu}(\bfw;T))\\
 &\geq& \sigma_{\min}(\nabla F (\bfw ;T))- \|\nabla F(\bfw;T)-\nabla F_{\mu}  (\bfw;T)\| &(\text{by}~\eqref{lower-bd-HH})  \\
&=& \sigma_{\min}(\nabla F(\bfw ;T))- \mu &(\text{by}~\eqref{sta-eq-1}) \\
&\geq& \sigma_{\min}(\nabla F(\bfw^*;T))- \|\nabla F (\bfw^*;T)-\nabla F (\bfw;T)\|- \mu   &(\text{by}~\eqref{lower-bd-HH}) \\
&\geq&\sigma_{\min}(\nabla F(\bfw^*;T))- c_*/2 - \mu&(\text{by}~\eqref{HH*2}) \\
&\geq&\sigma_{\min}(\nabla F(\bfw^*;T))-  c_* &(\text{by}~\mu\leq c_*/2)   \\
&\geq&c_*. &(\text{by}~\eqref{JJogG0})  
\end{array}
\end{eqnarray}
Similarly, we also have
\allowdisplaybreaks 
\begin{eqnarray}\label{upper-bd-FF}
\arraycolsep=1.4pt\def\arraystretch{1.5}
\begin{array}{rclr}
&&\|\nabla F_{\mu}(\bfw;T)\|\\
 &\leq& \|\nabla F (\bfw ;T)\|+ \|\nabla F(\bfw;T)-\nabla F_{\mu}  (\bfw;T)\| \\
&=& \|\nabla F (\bfw ;T)\|+ \mu &(\text{by}~\eqref{sta-eq-1})\\
&\leq& \|\nabla F(\bfw^*;T)\|+ \|\nabla F (\bfw^*;T)-\nabla F (\bfw;T)\|+ \mu   &(\text{by}~\eqref{lower-bd-HH})  \\
&\leq& \|\nabla F(\bfw^*;T)\|+ c_*/2 + \mu&(\text{by}~\eqref{HH*2}) \\
&\leq& \|\nabla F(\bfw^*;T)\|+ c_* &(\text{by}~\mu\leq c_*/2)   \\
&\leq& C_*/2+c_* &(\text{by}~\eqref{JJogG1}) \\
&\leq& C_*. &(\text{by}~c_*\leq C_*/2)   
\end{array}
\end{eqnarray}
The whole proof is completed. 
\end{proof}

%

Now we are ready to claim the following local quadratic convergence.

\begin{theorem}\label{quadratic-convergence} 
Let $\bfw^*$ be any {\ts} with $0<\tau<\tau_*$ of  (\ref{ell-0-1}), $\tau_*$ and $\delta_2^*, c_*, C_*$ be given by (\ref{tau-*}) and (\ref{radius-2}). Assume  \cref{ass} and \cref{ass2}. Let $\{\bfw^k\}$ be the  sequence generated by  \cref{Alg-NM01} and $0\leq\mu_{-1}\leq c_*/2$. If the initial point  satisfies $\bfw^0\in N(\bfw^*,\delta_*)$, where
\begin{eqnarray}\label{delta-*}
\delta_*:=\min\left\{\delta_2^*,~{c_*}/{(2 (L_*+2\rho C_*))}\right\},
\end{eqnarray}
 then the following results hold.
\begin{itemize}
\item[a)] The sequence $\{\bfd^k\}_{k\geq0}$ is well defined and $\lim_{k\rightarrow\infty}\bfd^k=0$. 
\item[b)] The whole sequence $\{\bfw^k\}$ converges to $\bfw^*$ quadratically, namely,
\beq
\|\bfw^{k+1}-\bfw^{*}\|  &\leq& 
  ((L_*+2\rho C_*)/c_*) \|\bfw^{k}-\bfw^{*} \|^2. \nonumber 
\eeq
\item[c)] The halting condition satisfies
\beq
\|F(\bfw^{k+1};T_{k+1})\| 
&\leq& (L_*+2\rho C_*)(C_*/c_*^{3})  \|F(\bfw^{k};T_k)\|^2  \nonumber 
\eeq
and  \Cref{Alg-NM01} reaches $\|F(\bfw^{k};T_k)\| <  \epsilon$ for a given tolerance $\epsilon>0$ when \beq\label{k-tol}
k\geq \Big\lceil \log_2\Big(2 \sqrt{(L_*+2\rho C_*)(C_*/c_*)^3  }\|\bfw^{0}-\bfw^{*} \|\Big)-\log_2(\sqrt{\epsilon})\Big\rceil.\eeq  
\end{itemize}
\end{theorem}
\begin{proof} a) It is easily observed   
\begin{eqnarray}\label{mu-k}
0< \mu_{k} \leq \mu_{k-1}, ~k=1,2,3,\ldots, ~~~~{\rm and}~~~~ \lim_{k\rightarrow\infty}\mu_k=0.
\end{eqnarray} 
It follows from \cref{lemma-neighbour-1} and the facts
$\delta_*\leq \delta_2^*\leq\delta_1^*$ and $\bfw^0\in N(\bfw^*,\delta_*)$ that     
\beq\label{F-w*-T0-0}F (\bfw^{*};T_0) =0  
\eeq
with $T_0=\S_0\cup\E_0^o$ and from \cref{lemma-neighbour-2} that     
\beq\label{grad-F-w*-T0-c} C_*  \geq \|\nabla F_{\mu_0}(\bfw^0;T_0)\|\geq\sigma_{\min}( \nabla F_{\mu_0}(\bfw^0;T_0)) \geq  c_*
\eeq
Here, we used the fact that $\mu_0\leq \mu_{-1}\leq c_*/2$ by  \eqref{mu-k}. 
From \eqref{newton-dir-0}, we have
\beq\label{dkt}
\nabla F_{\mu_0} (\bfw^{0};T_0)~\bfd^{0}=- F (\bfw^{0};T_0).
\eeq
  \cref{lemma-neighbour-2} states that $\nabla F_{\mu_0} (\bfw^{0};T_0)$ is non-singular and thus $\bfd^{0}$ is well defined.
Let 
\beq\label{w0-beta-w*}
\bfw^{0}_\beta= \bfw^{*} + \beta(\bfw^{0}-\bfw^{*})=(\bfx^{0}_\beta;\bfz^{0}_\beta)
\eeq
where $\beta\in[0,1].$ One can easily check that $\bfw^{0}_\beta\in N(\bfw^*,\delta_*)$ as $$\|\bfw^{0}_\beta-\bfw^{*}\|= \beta\|\bfw^{0}-\bfw^{*}\|\leq  \delta_*.$$
The definition  in \eqref{sta-eq-1} enables us to obtain
\allowdisplaybreaks 
\begin{eqnarray*}
\arraycolsep=1.4pt\def\arraystretch{1.5}
\begin{array}{rcll}
&&\|\nabla F_{\mu_0} (\bfw^{0};T_0)- \nabla F (\bfw^{0}_\beta;T_0) \|  \\
&\leq&\|\nabla^2 f(\bfx^0)- \nabla^2 f(\bfx^0_\beta)\| + \mu_0&(\text{by}~\eqref{sta-eq-1})  \\
& \leq& L_* \|\bfx^0-\bfx^0_\beta\|+ \rho  \|F(\bfw^0;T_0)\|&(\text{by}~\eqref{update-mu}) \\ 
& \leq& L_* \|\bfw^0-\bfw^0_\beta\|+ \rho \|\nabla F_{\mu_0} (\bfw^{0};T_0)\| \|\bfd^0\|&(\text{by}~\eqref{newton-dir-0}) \\ 
& \leq& L_* \|\bfw^0-\bfw^0_\beta\|+ \rho C_* \|\bfw^1-\bfw^0\|&(\text{by}~\eqref{grad-F-w*-T0-c}) \\ 
& \leq& L_*(1-\beta)\|\bfw^0-\bfw^*\|+ \rho C_* (\|\bfw^1-\bfw^*\|+\|\bfw^0-\bfw^*\|) &(\text{by}~\eqref{w0-beta-w*})  \\ 
& =& (L_*(1-\beta)+\rho C_*)\|\bfw^0-\bfw^*\|+ \rho C_* \|\bfw^1-\bfw^*\|.
\end{array} \end{eqnarray*}
Denote $\Theta_{\mu_0}:=\nabla F_{\mu_0} (\bfw^{0};T_0)$ and $\Theta(\beta):= \nabla F (\bfw^{0}_\beta;T_0)$.  The above condition yields
  \beq\label{lipschitz}
\arraycolsep=1.4pt\def\arraystretch{1.5}
\begin{array}{rcll}
 \int_0^1 \|\Theta_{\mu_0}  - \Theta(\beta)\|d\beta
\leq (L_*/2+\rho C_*)\|\bfw^0-\bfw^*\|+ \rho C_* \|\bfw^1-\bfw^*\|.
\end{array} \eeq
For the fixed $T_0$, the function $F(\cdot,T_0)$ is differentiable, which by \eqref{F-w*-T0-0} derives
\beq\label{mean-value}
\arraycolsep=1.4pt\def\arraystretch{1.5}
\begin{array}{rcll}
F (\bfw^{0};T_0)=F (\bfw^{*};T_0)+\int_0^1 \Theta(\beta)(\bfw^{0}-\bfw^{*})d\beta=\int_0^1 \Theta(\beta) (\bfw^{0}-\bfw^{*})d\beta.
\end{array}\eeq
Now the following chain of inequalities holds.
\begin{eqnarray*}
\arraycolsep=1.4pt\def\arraystretch{1.5}
\begin{array}{rcll}
c_*\|\bfw^{1}-\bfw^{*}\| &=& c_*\|\bfw^{0}+\bfd^{0}-\bfw^{*}\|\\
& {=} &c_*\|\bfw^{0}-\bfw^{*}-\Theta_{\mu_0}^{-1}   F(\bfw^{0};T_0) \|&~~(\text{by}~\eqref{dkt})\\
& {\leq} & \|\Theta_{\mu_0} (\bfw^{0}-\bfw^{*})-   F(\bfw^{0};T_0) \|&~~(\text{by}~\eqref{grad-F-w*-T0-c})\\
&= & \|\Theta_{\mu_0}  (\bfw^{0}-\bfw^{*})-  \int_0^1 \Theta(\beta)(\bfw^{0}-\bfw^{*})d\beta \| &~~(\text{by}~\eqref{mean-value}) \\
& \leq& \int_0^1 \|\Theta_{\mu_0} -\Theta(\beta) \|\|\bfw^{0}-\bfw^{*}\|d\beta \\
& =& ({\theta_*}/2) \|\bfw^{0} -\bfw^{*}\|^2+   {\rho C_*}  \|\bfw^1-\bfw^*\|\|\bfw^{0}-\bfw^{*}\| &~~(\text{by}~\eqref{lipschitz}) \\
&\leq& ({\theta_*}/2) \|\bfw^{0} -\bfw^{*}\|^2+ {\rho C_*\delta_*}  \|\bfw^1-\bfw^*\|   \\
&\leq& ({\theta_*}/2) \|\bfw^{0} -\bfw^{*}\|^2+ (c_*/2) \|\bfw^1-\bfw^*\|, 
\end{array}
\end{eqnarray*}
where $\theta_*:=L_* +2\rho C_*$ and the last inequality is from $\delta_*\leq  {c_*}/(2 \theta_*)$ by \eqref{delta-*} and
\beq\label{w0-w*-delta*}
\|\bfw^{0}-\bfw^{*}\|<\delta_*\leq  {c_*}/(2 \theta_*)<  {c_*}/({2\rho  C_*}).
\eeq
 The above chain of inequalities  suffices to the following fact
\beq\label{quadratic-chain}
\|\bfw^{1}-\bfw^{*}\|\leq  (\theta_*/c_*) \|\bfw^{0} -\bfw^{*}\|^2.
\eeq
This together with $\|\bfw^{0}-\bfw^{*}\|<\delta_*$ and \eqref{w0-w*-delta*} derives
\beq 
\|\bfw^{1}-\bfw^{*}\|  &\leq& 
   (\theta_*/c_*)  \delta_* \|\bfw^{0}-\bfw^{*} \|\leq (1/2)\|\bfw^{0}-\bfw^{*} \|< \delta_*,  \nonumber
 \eeq
which means $\bfw^1\in N(\bfw^*,\delta_*)$. In addition, $\mu_1\leq\mu_0\leq c_*/2$ by  \eqref{mu-k}. Hence, replacing $T_0 $ by $T_1$, the same reasoning allows us to show that $\bfd^{1}$ is well defined and $$\|\bfw^{2}-\bfw^{*}\|  \leq (\theta_*/c_*)   \|\bfw^{1}-\bfw^{*} \|^2 .$$ By the induction, we can conclude that $\bfw^k\in N(\bfw^*,\delta_*)$, $\bfd^{k}$ is well defined and 
\beq\label{quadratic-rate}
\hspace{10mm}\|\bfw^{k+1}-\bfw^{*}\|  &\leq& (\theta_*/c_*) \|\bfw^{k}-\bfw^{*} \|^2, \\
\label{quadratic-rate-1}&\leq& (\theta_*/c_*)  \|\bfw^{k}-\bfw^{*} \|\delta_* \leq (1/2) \|\bfw^{k}-\bfw^{*} \|.\hspace{.2cm}(\text{by}~\eqref{w0-w*-delta*})
\eeq
Therefore, \eqref{quadratic-rate} claims b). The conclusion of a) can be made by \eqref{quadratic-rate-1}  that   $$\bfw^{k}\rightarrow\bfw^{*},~~ \bfd^k=\bfw^{k+1}-\bfw^{k}=\bfw^{k+1}-\bfw^{*}+\bfw^{*}-\bfw^{k}\rightarrow0.$$ 

c)  The  above  proof shows $\bfw^k\in N(\bfw^*,\delta_*)$ and hence \eqref{gw*-0} results in
\beq\label{wk-in-N}  F(\bfw^{*};T_k){=}0,\eeq
where $T_k=\S_k\cup\E_k^o$.
By letting $\bfw^{k}_\beta= \bfw^{*} + \beta(\bfw^{k}-\bfw^{*}),$ where $\beta\in[0,1]$, we have $\bfw^k_\beta\in N(\bfw^*,\delta_*)$. To show \eqref{bd-jacobian-F-w-T} in \cref{lemma-neighbour-2}, we verified the lower and upper bounds by \eqref{lower-bd-FF} and \eqref{upper-bd-FF}. Similarly, we can prove these bounds hold for $\nabla F (\bfw^{k}_\beta;T_k)$. (In fact, since $\bfw^k_\beta\in N(\bfw^*,\delta_*)\subseteq N(\bfw^*,\delta_2^*)$ by $\delta_*\leq \delta_2^*$, one just needs to set $\mu=0$ and $\bfw= \bfw^{k}_\beta$ in  \eqref{lower-bd-FF} and \eqref{upper-bd-FF}. Therefore,
\beq\label{C-F-w-beta-c}
C_*  \geq \|\nabla F (\bfw^{k}_\beta;T_k)\|\geq\sigma_{\min}( \nabla F (\bfw^{k}_\beta;T_k)) \geq  c_*, 
\eeq
Again, the function $F(\cdot,T_k)$ is differentiable for the fixed $T_k$, so the Mean-value theorem states that there is a $\beta_0\in(0,1)$ satisfying
\beq\label{FkTk}
\|F(\bfw^{k},T_k)\|&=&\|F(\bfw^{*};T_k)+  \nabla F (\bfw^{k}_{\beta_0};T_k) (\bfw^{k}-\bfw^{*})\|\nonumber\\
&=&\| \nabla F (\bfw^{k}_{\beta_0};T_k) (\bfw^{k}-\bfw^{*}) \|\hspace{1.7cm}(\text{by}~\eqref{wk-in-N})\nonumber\\
&\in&  [~c_*\|\bfw^{k}-\bfw^{*} \|, C_*\|\bfw^{k}-\bfw^{*} \|~], \hspace{0.8cm}(\text{by}~\eqref{C-F-w-beta-c})
\eeq
This contributes to 
\allowdisplaybreaks
\begin{eqnarray*}\arraycolsep=1.4pt\def\arraystretch{1.5}
\begin{array}{rclr}
\|F(\bfw^{k};T_{k})\| \leq  C_*\|\bfw^{k}-\bfw^{*}\|  
 & {\leq} & (\theta_*C_*/c_*)   \|\bfw^{k-1}-\bfw^{*} \|^2 &(\text{by}~\eqref{quadratic-rate}) \\
 & {\leq}& (\theta_*C_*/c_*^{3})  \|F(\bfw^{k-1};T_{k-1})\|^2&(\text{by}~\eqref{FkTk})\\
 &{\leq}& (\theta_*C_*^3/c_*^{3})      \|\bfw^{k-1}-\bfw^{*} \|^2&(\text{by}~\eqref{FkTk})\\
& {\leq}& (\theta_*C_*^3/c_*^{3})  2^{-2}  \|\bfw^{k-2}-\bfw^{*} \|^2&~~(\text{by}~\eqref{quadratic-rate-1})\\
    &\vdots&\\
 & \leq& (\theta_*C_*^3/c_*^{3})   2^{2-2k} \|\bfw^{0}-\bfw^{*} \|^2,&(\text{by}~\eqref{quadratic-rate-1})
 \end{array}  \end{eqnarray*}
where the third inequality yields the first conclusion in c). This also enable to verify that $\|F(\bfw^{k};T_k)\|$ $< {\epsilon}$ if $k$ satisfies \eqref{k-tol}.
The whole proof is completed.
\end{proof} 

\begin{remark} \label{Remark-Primal-Dual}
 {\bf Relationship to primal-dual active-set algorithms}.
It is interesting to note that when $\mu_k \equiv \mu$ (a constant) for all indices $k$, \cref{Alg-NM01} shares a similar framework to the primal-dual
active-set algorithm in \cite[Alg.~1]{fan2014primal}, whose main target is
the convex quadratic programming in compressed sensing with $\ell_1$ regularization. 
In terms of convergence theory, both \cref{quadratic-convergence} and
\cite[Thm.~2]{fan2014primal} require the initial point to be close to the
interested solution point.
There are two key differences. (i) \cref{quadratic-convergence}
is able to identify the quadratic convergence region $N(\bfw^*, \delta_*)$ 
with $\delta_*$ being given by (\ref{delta-*}), while
\cite[Thm.~2]{fan2014primal} does not have such a characterization and
is only about it local convergence (not its convergence rate).
(ii) However,  \cite[Thm.~2]{fan2014primal}  can be globalized via a continuation
technique, while it is challenging to globalize \cref{Alg-NM01} because
we are dealing with $0/1$-loss function and there are no merit functions available
for globalization.
\end{remark} 
\begin{remark}{\bf On the choice of the smoothing parameter} $\mu_k$.
 {An interesting question raised by one referee is whether the particular
choice of $\mu_k$ in (\ref{update-mu}) may play a role in globalization of 
\cref{Alg-NM01}.
From the smoothing perspective, there exists a number of good strategies to
update $\mu$ as long as it drives $\mu_k \rightarrow 0$. For example,
we may update $\mu_k$ by solving the equation $e^\mu - 1=0$ via Newton's method
as done in \cite{qi2000smoothing}, see also \cite{huang2006smoothing} for other
options. 
To incorporate such a strategy in a globalization scheme, we must find a merit
function to work with. As commented in  \cref{Remark-Primal-Dual},
it is not easy to construct a merit function because the composition of the
operator $\| (\cdot)_+\|_0$ with the inequality constraint $A \bfx \le \bfb$
leads to the scenario where the sparsity is not over a symmetric set any more. 
We refer to \cite{beck2016minimization, lu2015optimization} for
detailed discussion on algorithmic advantages of sparsity being
over symmetric sets.
}	
\end{remark}

\section{Numerical Experiments}\label{sec:numerical}

In this part, we will conduct extensive numerical experiments of  {\tt NM01} in  \cref{Alg-NM01} by using MATLAB (R2019a) on a laptop with $32$GB memory and Inter(R) Core(TM) i9-9880H 2.3Ghz CPU, against a few leading solvers for solving SVM and  1-bit CS problems. 

\subsection{Experiments for SVM}

There exists a large body of SVM literature. We only focus on the binary classification, which has a training dataset $\{(\bfa_i^0,c_i):i\in\N_m\}$, with $\bfa_i^0\in \R^{n-1}$ being samples and $c_i\in \{-1,1\}$ being the two classes. It is widely recognized that the data are often linearly inseparable and \eqref{ell-0-1} is an ideal model to deal with this case with the following setup 
\[
  f(\bfx)=\|D\bfx\|^2,~~A=-[c_1\bfa_1,\ldots,c_m\bfa_m]^\top,~~\bfb=\bfe,
\]
where $D$ is a diagonal matrix with $D_{ii}=1,i\in\N_{n-1}$ and $D_{nn}\geq0$ (e.g., $D_{nn}=10^{-4}$),  {and $\bfa_i=(\bfa_i^0 ;1)\in\R^n,i\in\N_m$}. We will consider two types of datasets: synthetic data and real data described below.

\begin{example}[Synthetic data in $\R^2$]\label{ex:syn-data-outlier} Give four samples $(0,0), (0,1), (1,0), (1,a)$ with labels $+1,+1,-1,-1$, where the last point can be treated as an outlier when $a>1$.
\end{example}

\begin{example}[Real data in higher dimensions]\label{ex:real-data} We select 40 datasets from three libraries: libsvm
, uci
and kaggle.
All datasets are feature-wisely scaled to $[-1,1]$ and all the classes not being  $1$ are treated as $-1$.  Their details are presented in  \cref{Table-svm-less-m-more-n}. There are 16 datasets with $m \leq n$ and 24 datasets with $m > n$.  
\end{example}

\begin{table}[!th]
	\renewcommand{\arraystretch}{.8}\addtolength{\tabcolsep}{1pt}
	\caption{Descriptions of real datasets.}\vspace{-3mm}
	\label{Table-svm-less-m-more-n}
	\begin{center}
		\begin{tabular}{lllrrrr }
			\hline
Data&Datasets&	Source	&	$n$	&	$m$	&	Sparse	&		\\\hline
\multicolumn{7}{c}{$m\leq n$}\\\hline
\texttt{arce}&	Arcene	&	uci	&	10000 	&	100 &		No	\\
\texttt{colc}&	Colon-cancer	&	libsvm	&	2000 	&	62 	&		No	\\
\texttt{dbw1}&	\multirow{4}{3.5cm}{Dbworld e-mails}	&	uci	&	4702 	&	64 	&		Yes	\\
\texttt{dbw2}&		&	uci	&	3721 	&	64 	&		Yes	\\
\texttt{dbw3}&		&	uci	&	242 	&	64 	&		Yes	\\
\texttt{dbw4}&		&	uci	&	229 	&	64 	&		Yes	\\
\texttt{dext}&	Dexter	&	uci	&	19999 	&	300 	&		Yes	\\
\texttt{dmea}&	Detect malacious executable 	&	uci	&	531 	&	373 	&		Yes	\\
\texttt{doro}&	Dorothea	&	uci	&	100000 	&	800 	&		Yes	\\
\texttt{dubc}&	Duke breast-cancer	&	libsvm	&	7129 	&	38 	&		No	\\
\texttt{fabc}&	Farm ads binary classification	&	kaggle	&	54877 	&	4143 	&		Yes	\\
\texttt{leuk}&	Leukemia	&	libsvm	&	7129 	&	38 	&		No	\\
\texttt{lsvt}&	Lsvt voice rehabilitation	&	uci	&	310 	&	126 	&		No	\\
\texttt{newb}&	News20.binary	&	libsvm	&	{1355191} 	&	19996 	&		Yes	\\
\texttt{rcvb}&	Rcv1.binary	&	libsvm	&	47236 	&	20242 	&		Yes	\\
\texttt{scad}&	Scadi	&	uci	&	205	&	70 &		No	\\\hline
\multicolumn{7}{c}{$m> n$}\\\hline

\texttt{aips}&	Airline passenger satisfaction	&	kaggle	&	22 	&	103904 	&	 	No	\\

\texttt{ccfd}&	Credit card fraud dtection	&	kaggle	&	28 	&	284807 	&	 	No	\\

\texttt{covt}&	Covtype.binary	&	libsvm	&	54 	&	581012 	&	 	Yes	\\

\texttt{dccc}&	Default of credit card clients	&	kaggle	&	23 	&	30000 	&	 	No	\\

\texttt{escd}&	Email spam classification dataset	&	kaggle	&	3000 	&	5172 	 	&	Yes	\\

\texttt{gise}&	Gisette	&	libsvm	&	5000 	&	6000 	&	 	Yes	\\

\texttt{hepm}&	Hepmass	&	uci	&	28 	&	{7000000} 	&	 	No	\\

\texttt{hfxf}&	Hedge fund x: financial mod. chal.	&	kaggle	&	88 	&	10000 	&	 	No	\\

\texttt{higg}&	Higgs	&	uci	&	28 	&	{11000000} 	&	 	No	\\

\texttt{hmeq}&	Hmeq\_data	&	kaggle	&	10 	&	5960 	&	 	No	\\
 
\texttt{htru}&	Htru2	&	uci	&	8 	&	17898 	&	 	No	\\

\texttt{idac}&	Ida2016challenge	&	uci	&	170 	&	60000 	&	 	Yes	\\
   
\texttt{ijcn}&	Ijcnn1	&	libsvm	&	22 	&	49990 	&	 	Yes	\\
   
\texttt{mrpe}&	Malware analysis datasets: raw pe	&	kaggle	&	1024 	&	51959 	&	 	No	\\

 \texttt{mtpe}&	Malware analysis datasets: top-1000 	&	kaggle	&	1000 	&	47580 	&	 	Yes	\\

\texttt{ospi}&	Online shoppers purchasing intention	&	uci	&	17 	&	12330 	&	 	No	\\

 \texttt{pssr}&	Parkinson speech dataset	&	uci	&	26 	&	1039 	&	 	No	\\
 
 \texttt{qsot}&	Qsar oral toxicity	&	uci	&	1024 	&	8992 	&	 	Yes	\\

 \texttt{reas}&	Real-sim	&	libsvm	&	20958 	&	72309 	&	 	Yes	\\

 \texttt{retb}&	Real time bidding	&	kaggle	&	88 	&	{1000000} 	&	 	No	\\

\texttt{sctp}&	Santander customer transaction	&	kaggle	&	200 	&	200000 	&	 	No	\\

 \texttt{skin}&	Skin\_nonskin	&	libsvm	&	3 	&	245056 	&	 	No	\\

 \texttt{spli}&	Splice	&	libsvm	&	60 	&	1000 	&	 	No	\\
 
 \texttt{susy}&	Susy	&	uci	&	18 	&	{5000000} 	&	 	No	\\ 
 \hline
 		\end{tabular}
	\end{center}
\end{table}

There are large numbers of methods that have been proposed for SVMs, each with its advantages/disadvantages. It is more reasonable to compare {\tt NM01} with those methods that aim at optimizing the $0/1$-loss function directly, such as MIP-based methods. However, it is known that MIP-based methods prefer the datasets on small scales (see the numerical experiments reported in \cite{nguyen2013algorithms,tang2014mixed,ustun2016supersparse}) and behave very slowly for the datasets on mediate/large scales, such as most datasets in \Cref{Table-svm-less-m-more-n}.    Therefore, we will not include them in the following numerical comparisons.

On the other hand, there is very limited work on developing methods that directly optimize $0/1$-loss from the perspective of continuous optimization. Because of this, we are unable to find an available Matlab implementation for such kinds of methods. Hence, we only select five leading solvers, with available Matlab implementations from the machine learning community. These methods solve the surrogate/relaxations of $0/1$-loss involved SVMs. They are {\tt HSVM} from the library {\tt libsvm}\footnote{\url{https://www.csie.ntu.edu.tw/~cjlin/libsvm/}}\cite{chang2011libsvm}, {\tt SSVM} \cite{SV1999} implemented by liblssvm\footnote{\url{https://www.esat.kuleuven.be/sista/lssvmlab/}}\cite{pelckmans2002matlab},   {\tt RSVM} \cite{wu2007robust},   {\tt LSVM}  from the library liblinear\footnote{\url{https://www.csie.ntu.edu.tw/~cjlin/liblinear/}}\cite{fan2008liblinear}, and {\tt FSVM} (a MATLAB built-in function {\tt fitclinear}\footnote{\url{https://mathworks.com/help/stats/fitclinear.html}}). All involved parameters are set as their default values.  To demonstrate the performance of one method, let $\bfx$ be its obtained solution  and $A_0:=[\bfa_1,\ldots,\bfa_m]^\top$. We will report the CPU  {\tt Time} and the classification accuracy {\tt Acc} defined by ${\tt Acc}:=1-  \|{\rm sgn}(A_0\bfx)-\bfc\|_0/m.$

\subsubsection{Implementation of \cref{Alg-NM01}} 

We terminate our algorithm if one of the conditions is satisfied: $k\geq 1000$ or $\|F(\bfw^k;T_k)\|<10^{-4}$.  We initialize $\bfx^0=0$ and $\bfz^0=\bfe$, and set  $\mu_{-1}=0.05$ if $m<n$ and $\mu_{-1}=5$ otherwise.  Moreover,  we update $\mu_k$  by \eqref{update-mu} with $\rho=1$ and $\alpha=0.5$ if $k$ is a multiple of $5$.
 {The rest of this part is about setting the parameters $\tau$ and $\lambda$.
We try to suggest general principles, but bearing in mind that the best 
strategy of setting $\tau$ and $\lambda$ is problem dependent.
For the  validation purpose, we conducted the performance comparison of
\Cref{Alg-NM01} on four test problems
{\tt arce, colc, dbw1} and {\tt fabc} where we vary one parameter 
while the other is
being fixed. Fig.~\ref{fig:tau} is for fixed $\lambda$ and
Fig.~\ref{fig:lam} is for fixed $\tau$.
}  

\begin{itemize} 
	
\item[(i)] {\bf For SVM problems}.
 {It follows from \eqref{lambda-lower-bound} in \Cref{rem-tau-lam} that  if $2\lambda \tau \leq \min_i\{b_i^2: b_i>0\}$,
then $\bfx^*=0$ and $\bfz^*=0$ is a P-stationary point. 
For SVM, this condition turns into $2\lambda \tau \leq 1$ since $\bfb=\bfe$. 
This phenomenon can be observed in our numerical experiments. 
For example, zero solutions were obtained by \Cref{Alg-NM01} when $\tau\leq 1/(2\lambda)$ for fixed $\lambda=15$ in Fig.~\ref{fig:tau} 
and when $\lambda\leq 1/(2\tau)$ for fixed $\tau=5$ in Fig.~\ref{fig:lam}. 
Hence, it is recommended to set $\tau$ and $\lambda$ 
to satisfy $2\lambda \tau > 1$ for SVM problems.}

\item[(ii)] {\bf On the choice of } $\tau$.
 {Despite that a sufficient condition  $0<\tau<\tau_*$ is provided in \cref{quadratic-convergence},  it is still difficult to set a proper $\tau$ as $\tau_*$ is not known. However, as the condition is sufficient, it is unnecessary to choose it from $(0, \tau_*)$ strictly. To see its effect, we tested it with varying $\tau\in [10^{-3},10]$, fixed $\lambda=15$  and report the results in Fig.~\ref{fig:tau}.  It can be clearly seen that bigger values of $\tau$  (e.g., $\tau\geq 1$) lead to better accuracy {\tt ACC}.  
An underlying heuristic explanation is as follows:   
\Cref{Alg-NM01} solves the system  \eqref{newton-dir-0} with index set $T_k={\cal S}_k\cup {\cal E}_k^o$ being decided by the parameter $\tau$, see \eqref{p-stationary-T}. 
We observed that setting $\tau$ too small often led to infrequent change of
$T_k$ and this often forced the algorithm fell into (possibly undesirable) local regions too quickly.  
By contrast, setting $\tau$ slightly bigger enabled altering $T_k$ 
frequently enough to make \Cref{Alg-NM01} escape from undesirable local regions so as to achieve better solutions. 
Since the  theoretical convergence is in favour of small values of $\tau<\tau_*$,  it is not suggested to set the  values of $\tau$  too large.}  

\item[(iii)] {\bf On the choice of } $\lambda$.
 { For the parameter $\lambda$, we varied values $\lambda\in[10^{-2},10^2]$ and report its effect in Fig.~\ref{fig:lam}. 
As expected, zero solutions were achieved when $\lambda\leq 1/(2\tau)$. 
From the left sub-figure, ACCs are in favour of bigger values of $\lambda>1/(2\tau)$ which is reasonable since it penalizes the 0/1 loss  in \eqref{ell-0-1}. This choice of $\lambda$ is consistent with what we have
observed in (i) above.}

\item[(iv)] {\bf Estimating} $\tau_*$.
 {Although $\tau_*$ is unknown, we may be able to numerically
estimate it by using the obtained solution and \eqref{tau-*} provided that some
information was available a priori.  
Note that if $\bfx^*=0$ and $\bfz^*=0$  then $\tau_*=\min\{\tau_1,\tau_2\} = 1/(2\lambda)$ by \eqref{tau-*}, which can be seen in   Fig.~\ref{fig:tau} and  Fig.~\ref{fig:lam}. On the other hand, if a solution satisfies that $A\bfx^*+\bfb<0$, then $\tau_1=\tau_2=\infty$ \eqref{tau-*} and hence $\tau_*=\infty$. This phenomena can be observed for datasets {\tt arce} and {\tt colc} since they are linearly separable, see the results for $\tau>1/(2\lambda)$ in  Fig.~\ref{fig:tau} and for $\lambda>1/(2\tau)$ in  Fig.~\ref{fig:lam}. 
}

\end{itemize} 
 
 {
Therefore, in the following experiments, we set $\tau=5$ and $\lambda=15$ for simplicity. 
}

 \begin{figure}[!th]
\centering
	\includegraphics[width=1\textwidth]{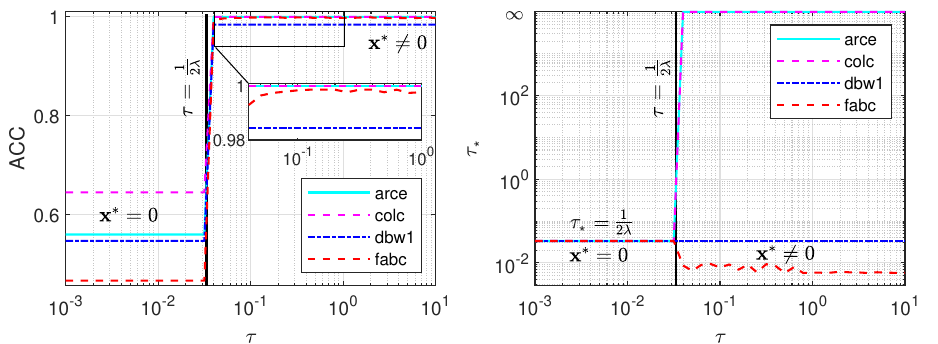}
\vspace{-5mm}
\caption{Effect of $\tau$ with fixed $\lambda=15$ for SVM.}\vspace{-5mm}
\label{fig:tau}
\end{figure}

 \begin{figure}[!th]
\centering
	\includegraphics[width=1\textwidth]{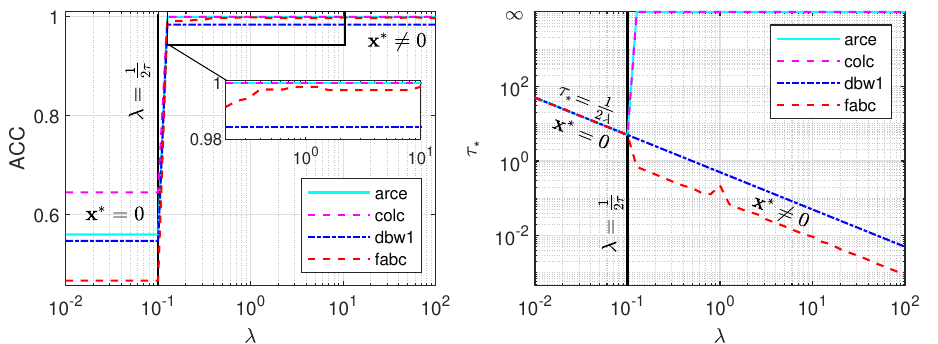}
\vspace{-5mm}
\caption{Effect of $\lambda$ with fixed $\tau=5$ for SVM.}\vspace{-5mm}
\label{fig:lam}
\end{figure}

\subsubsection{Numerical comparisons} 

We first employ five methods to solve \cref{ex:syn-data-outlier} under different $a=1,10,100$ to test their robustness to the outliers. For such data, the classifier with a maximum margin is $x_1^*=1/2$. The classifiers by each method are plotted in Fig.~\ref{fig:ex0}, where {\tt HSVM} is omitted since it solves the dual problem and does not provide the solution $\bfx$. Obviously, {\tt NM01}  finds the true classifiers for all scenarios, while the other methods are influenced significantly by $a$. 

\begin{figure}[!th]
\centering
\begin{subfigure}{.325\textwidth}
	\centering
	\includegraphics[width=1\textwidth]{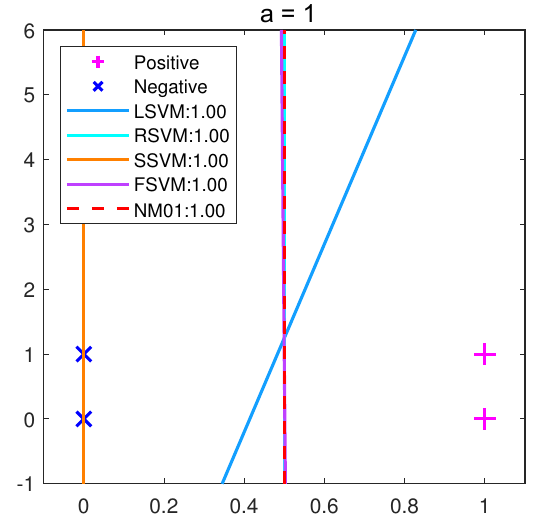}
\end{subfigure}%
\begin{subfigure}{.325\textwidth}
	\centering
	\includegraphics[width=1\textwidth]{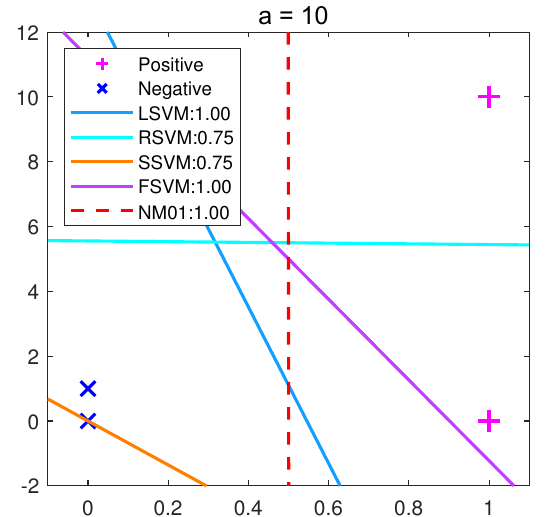}
\end{subfigure}  
\begin{subfigure}{.325\textwidth}
	\centering
	\includegraphics[width=1\textwidth]{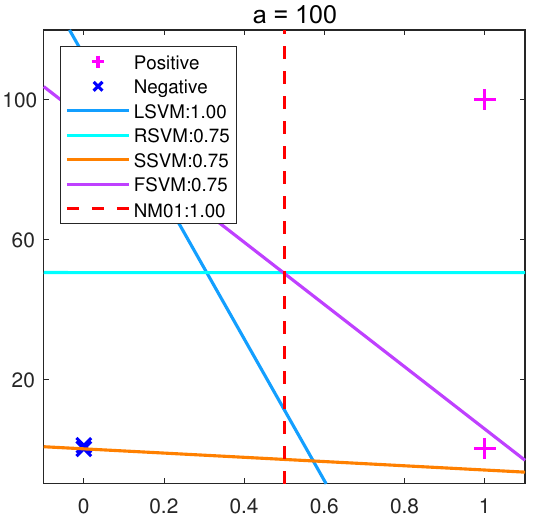}
\end{subfigure} 
 
\caption{Robustness to outliers.} \vspace{-5mm}
\label{fig:ex0}
\end{figure}

 For  \cref{ex:real-data}, we have 40 datasets with sample size from a few to ten million (e.g., {\tt higg} having $11,000,000$ samples).  Results of six methods are reported in \cref{table:ex2-no-test},  where ``$--$'' denotes the results are not obtained  if a solver takes too much time or requires a large memory  that is out of the capacity of our desktop.  For example, {\tt HSVM} consumes more than 10,000 seconds on  the data {\tt covt} and {\tt SSVM} requires at lest 32GB memory to solve {\tt mtpe}.  In general, {\tt NM01} renders the highest  {\tt Acc} for most datasets.   For the computational time, {\tt FSVM} and {\tt LSVM} are very fast for datasets of moderate sizes.  However, our method is more competitive especially when the data size  is in million scale, such as {\tt higg}, {\tt retb}, {\tt susy}, {\tt   hepm} with  more than $10^6$ samples, {\tt NM01} runs the fastest.  For instance, {\tt FSVM} and {\tt LSVM} respectively took 80.69 seconds and 65.65 seconds for {\tt higg}, which is solved by our method within 5.63 seconds.

 \begin{table}[!th]
\centering
\caption{Results of six solvers for  \cref{ex:real-data}}
\label{table:ex2-no-test} 
\renewcommand{\arraystretch}{.8}\addtolength{\tabcolsep}{-1pt}	
\begin{tabular}{l|cccccc|ccccccc}\hline
 &\multicolumn{6}{c|}{ {\tt Acc} }
&    &\multicolumn{6}{c}{ {\tt Time (seconds)}}\\
       \cline{2-7} \cline{8-14}
data&	{\tt FSVM}	&	{\tt HSVM}	&	{\tt LSVM}	&	{\tt RSVM}	&	{\tt SSVM}	&	{\tt NM01}	&	
&	{\tt FSVM}	&	{\tt HSVM}	&	{\tt LSVM}	&	{\tt RSVM}	&	{\tt SSVM}	&	{\tt NM01}\\ \hline
 {\tt arce 	}&	1.000 	&	1.000 	&	1.000 	&	1.000 	&	1.000 	&	1.000 	&	&	0.040 	&	0.965 	&	0.077 	&	0.017 	&	8.390 	&	0.051 	\\
{\tt colc 	}&	0.952 	&	1.000 	&	1.000 	&	1.000 	&	1.000 	&	1.000 	&	&	0.207 	&	0.018 	&	0.015 	&	0.493 	&	0.864 	&	0.008 	\\
 {\tt dbw1 	}&	0.984 	&	0.984 	&	0.984 	&	0.984 	&	0.984 	&	0.984 	&	&	0.043 	&	0.011 	&	0.002 	&	0.061 	&	7.713 	&	0.033 	\\
 {\tt dbw2 	}&	0.984 	&	0.984 	&	0.984 	&	0.984 	&	0.984 	&	0.984 	&	&	0.038 	&	0.010 	&	0.002 	&	0.069 	&	4.787 	&	0.023 	\\
 {\tt dbw3 	}&	0.984 	&	0.984 	&	1.000 	&	0.953 	&	1.000 	&	1.000 	&	&	0.022 	&	0.001 	&	$--$ 	&	0.035 	&	0.115 	&	0.002 	\\
 {\tt dbw4 	}&	0.984 	&	0.984 	&	1.000 	&	0.938 	&	1.000 	&	1.000 	&	&	0.094 	&	0.001 	&	0.001 	&	0.027 	&	0.124 	&	0.005 	\\
{\tt dext 	}&	1.000 	&	1.000 	&	1.000 	&	1.000 	&	1.000 	&	1.000 	&	&	0.008 	&	0.167 	&	0.009 	&	0.055 	&	81.63 	&	0.029 	\\
{\tt dmea 	}&	1.000 	&	1.000 	&	1.000 	&	0.984 	&	1.000 	&	1.000 	&	&	0.008 	&	0.010 	&	0.003 	&	0.518 	&	1.100 	&	0.013 	\\
{\tt doro 	}&	1.000 	&	1.000 	&	1.000 	&	1.000 	&	$--$ 	&	1.000 	&	&	0.026 	&	8.451 	&	0.078 	&	0.564 	&	$--$ 	&	0.163 	\\
{\tt dubc 	}&	1.000 	&	1.000 	&	1.000 	&	1.000 	&	1.000 	&	1.000 	&	&	0.010 	&	0.035 	&	0.019 	&	0.007 	&	5.969 	&	0.006 	\\
{\tt fabc 	}&	0.996 	&	0.999 	&	0.999 	&	0.994 	&	$--$ 	&	0.999 	&	&	0.040 	&	8.434 	&	0.194 	&	96.08 	&	$--$ 	&	0.275 	\\
{\tt leuk 	}&	1.000 	&	1.000 	&	1.000 	&	1.000 	&	1.000 	&	1.000 	&	&	0.010 	&	0.044 	&	0.022 	&	0.006 	&	5.991 	&	0.004 	\\
 {\tt lsvt 	}&	0.952 	&	0.984 	&	1.000 	&	0.873 	&	1.000 	&	1.000 	&	&	0.030 	&	0.007 	&	0.005 	&	0.045 	&	0.141 	&	0.008 	\\
 {\tt newb 	}&	0.995 	&	$--$ 	&	0.999 	&	$--$ 	&	$--$ 	&	0.999 	&	&	0.481 	&	$--$ 	&	2.031 	&	$--$ 	&	$--$ 	&	1.251 	\\
{\tt rcvb 	}&	0.990 	&	0.990 	&	0.997 	&	$--$ 	&	$--$ 	&	0.998 	&	&	0.092 	&	180.3 	&	0.256 	&	$--$ 	&	$--$ 	&	0.153 	\\
 {\tt scad 	}&	0.986 	&	1.000 	&	1.000 	&	0.971 	&	1.000 	&	1.000 	&	&	0.011 	&	0.001 	&	$--$ 	&	0.015 	&	0.097 	&	0.004 	\\\hline
 
{\tt aips 	}&	0.876 	&	0.877 	&	0.874 	&	$--$ 	&	$--$ 	&	0.878 	&	&	0.539 	&	509.6 	&	0.393 	&	$--$ 	&	$--$ 	&	0.028 	\\

 {\tt ccfd 	}&	0.999 	&	0.999 	&	0.999 	&	$--$ 	&	$--$ 	&	0.999 	&	&	7.155 	&	117.8 	&	1.451 	&	$--$ 	&	$--$ 	&	0.209 	\\
 {\tt covt 	}&	0.763 	&	$--$ 	&	0.757 	&	$--$ 	&	$--$ 	&	0.764 	&	&	8.638 	&	$--$ 	&	1.801 	&	$--$ 	&	$--$ 	&	0.444 	\\
 {\tt dccc 	}&	0.810 	&	0.809 	&	0.802 	&	$--$ 	&	0.799 	&	0.820 	&	&	0.098 	&	46.84 	&	0.079 	&	$--$ 	&	139.4 	&	0.011 	\\
 {\tt escd 	}&	0.993 	&	0.992 	&	0.996 	&	0.856 	&	0.971 	&	0.996 	&	&	0.077 	&	8.442 	&	0.113 	&	459.2 	&	13.09 	&	0.228 	\\
{\tt gise 	}&	1.000 	&	1.000 	&	1.000 	&	1.000 	&	1.000 	&	1.000 	&	&	0.215 	&	63.15 	&	0.451 	&	106.3 	&	1238 	&	0.298 	\\
{\tt hepm 	}&	0.837 	&	$--$ 	&	0.836 	&	$--$ 	&	$--$ 	&	0.840 	&	&	16.72 	&	$--$ 	&	37.37 	&	$--$ 	&	$--$ 	&	2.789 	\\
 {\tt hfxf 	}&	0.589 	&	0.589 	&	0.589 	&	0.572 	&	0.588 	&	0.590 	&	&	0.065 	&	17.25 	&	0.065 	&	179.1 	&	6.480 	&	0.010 	\\
 {\tt higg 	}&	0.641 	&	$--$ 	&	0.641 	&	$--$ 	&	$--$ 	&	0.651 	&	&	80.69 	&	$--$ 	&	65.65 	&	$--$ 	&	$--$ 	&	5.631 	\\
 {\tt hmeq 	}&	0.860 	&	0.859 	&	0.860 	&	0.803 	&	0.862 	&	0.865 	&	&	0.030 	&	0.687 	&	0.004 	&	59.40 	&	1.618 	&	0.002 	\\
 {\tt htru 	}&	0.977 	&	0.977 	&	0.977 	&	$--$ 	&	0.971 	&	0.979 	&	&	0.038 	&	0.626 	&	0.016 	&	$--$ 	&	25.81 	&	0.006 	\\
{\tt idac 	}&	0.991 	&	0.992 	&	0.992 	&	$--$ 	&	$--$ 	&	0.992 	&	&	0.376 	&	61.73 	&	0.740 	&	$--$ 	&	$--$ 	&	0.193 	\\
{\tt ijcn 	}&	0.924 	&	0.924 	&	0.923 	&	$--$ 	&	$--$ 	&	0.931 	&	&	0.223 	&	39.26 	&	0.097 	&	$--$ 	&	$--$ 	&	0.025 	\\
 {\tt mrpe 	}&	0.948 	&	$--$ 	&	0.951 	&	$--$ 	&	$--$ 	&	0.951 	&	&	2.678 	&	$--$ 	&	9.470 	&	$--$ 	&	$--$ 	&	1.324 	\\
 {\tt mtpe 	}&	0.968 	&	0.984 	&	0.981 	&	$--$ 	&	$--$ 	&	0.984 	&	&	0.396 	&	275.5 	&	13.11 	&	$--$ 	&	$--$ 	&	2.185 	\\
 {\tt ospi 	}&	0.884 	&	0.884 	&	0.879 	&	$--$ 	&	0.873 	&	0.893 	&	&	0.066 	&	3.916 	&	0.021 	&	$--$ 	&	10.35 	&	0.006 	\\
{\tt pssr 	}&	0.641 	&	0.643 	&	0.654 	&	0.626 	&	0.647 	&	0.663 	&	&	0.216 	&	0.063 	&	0.006 	&	1.105 	&	0.089 	&	0.001 	\\
 {\tt qsot 	}&	0.946 	&	0.969 	&	0.967 	&	0.849 	&	0.945 	&	0.971 	&	&	0.058 	&	23.91 	&	0.126 	&	1865 	&	233.9 	&	0.378 	\\
 {\tt reas 	}&	0.989 	&	0.989 	&	0.994 	&	$--$ 	&	$--$ 	&	0.994 	&	&	0.320 	&	741.7 	&	0.505 	&	$--$ 	&	$--$ 	&	0.682 	\\
 {\tt retb 	}&	0.998 	&	$--$ 	&	0.998 	&	$--$ 	&	$--$ 	&	0.998 	&	&	18.77 	&	$--$ 	&	12.57 	&	$--$ 	&	$--$ 	&	0.541 	\\
 {\tt sctp 	}&	0.910 	&	$--$ 	&	0.909 	&	$--$ 	&	$--$ 	&	0.914 	&	&	1.205 	&	$--$ 	&	6.304 	&	$--$ 	&	$--$ 	&	0.358 	\\
 {\tt skin 	}&	0.929 	&	0.929 	&	0.924 	&	$--$ 	&	$--$ 	&	0.943 	&	&	0.172 	&	232.7 	&	0.109 	&	$--$ 	&	$--$ 	&	0.029 	\\
{\tt spli 	}&	0.839 	&	0.839 	&	0.840 	&	0.805 	&	0.840 	&	0.840 	&	&	0.036 	&	0.148 	&	0.015 	&	0.566 	&	0.100 	&	0.001 	\\
 {\tt susy 	}&	0.788 	&	$--$ 	&	0.787 	&	$--$ 	&	$--$ 	&	0.790 	&	&	18.38 	&	$--$ 	&	22.16 	&	$--$ 	&	$--$ 	&	1.916 	\\											
\hline 
\end{tabular}
\end{table}

\subsection{Simulations for 1-bit CS} 

The aim of 1-bit CS is to recover a sparse signal $\bfx$ from  $\bfc={\rm sgn}(A_0\bfx)$, where $A_0:=[\bfa_1,~\ldots,\bfa_m]^\top\in\R^{m\times n}$ and $c_i\in\{1,-1\}, i\in\N_m$.  The original optimization model for 1-bit CS \cite{BB08} takes the following form: 
\[   \min~\|\bfx\|_0, ~~{\rm s.t.} \quad c_{i}\langle \bfa_i, \bfx\rangle\geq 0,~i\in\N_m. \]   
Various relaxation methods have been proposed. Here we adopt the smoothing technique using the popular $\ell_q$ norm  ($0<q <1)$ to approximate the $\ell_0$ norm \cite{lai2013improved}  and use the $0/1$-loss function to deal with the constraints.  This leads to the 
model \eqref{ell-0-1}  with 
 $$f(\bfx)=\sum_{i=1}^n (x_i^2+ \varepsilon^2)^{q/2}, ~~A=-[c_1\bfa_1,\ldots,c_m\bfa_m],~~\bfb=\epsilon {\bf1},$$  
where $\varepsilon >0$, $\epsilon>0$.    Here, $\bfb=\epsilon {\bf1}$ is adopted from  \cite{dai2016noisy}. In our test, we set $q=0.5, \epsilon=0.05$ but update $\varepsilon$ by $\varepsilon_0=0.5$ and $\varepsilon_{k+1}=\varepsilon_k/2$. The test problems are taken from \cite{huang2018robust} and are described as follows.

\begin{example}\label{ex:cs-cor} 
	Rows of $A_0$ are the independent and identically distributed (iid) samples  of $\mathcal{N}(0,\Sigma)$ with $\Sigma_{ij}=v^{-|i-j|}, i,j\in\N_n$ and $v\in(0,1)$.  The nonzero entries of the ground truth $s$-sparse vector $\bfx^*\in\R^n$, namely, $\|\bfx^*\|_0\leq s$, are generated from the i.i.d.  samples of the standard Gaussian distribution $\mathcal{N}(0,1)$, followed by a  normalization of $\bfx^*$  to be a unit vector. Let $\bfc^*={\rm sgn}(A_0\bfx^*)$ and $\tilde{\bfc}={\rm sgn}(A_0\bfx^*+ \xi)$, where entries of the noise $\xi$  are the i.i.d. samples of $\mathcal{N}(0,0.1^2)$. Finally,  we randomly select $\lceil r m\rceil$ entries in $\tilde{\bfc}$ and flip their signs, and the flipped vector is denoted by  $\bfc$, where  $r$  is the flipping ratio. 
\end{example}

We report the CPU {\tt time},  the signal-to-noise ratio 
${\tt SNR}:=-10{\log}_{10}(\| \bfx -\bfx^*\|^2)$, the hamming error  ${\tt HE}:= \|{\rm sgn}(A_0\bfx)-\bfc^*\|_0/m$, and the hamming distance ${\tt HD}:= \|{\rm sgn}(A_0\bfx)-\bfc\|_0/m$,  where $\bfx$ is the solution obtained  by a method.  We note that larger {\tt SNR} (smaller  {\tt HE}, or smaller {\tt HD}) corresponds to better recovery. 

\subsubsection{Implementation and benchmark methods} 

The stopping criteria and the rule for updating $\mu_k$  for \cref{Alg-NM01} are the same as for the SVM case.    We initialize $\bfx^0=0$ and $\bfz^0=\bfe$.   We tested the algorithm under different choices of $\tau$ and $\lambda$,  and only report the numerical results with $\tau=1$ and $\lambda=1$, which satisfy $2\tau\lambda>\min_i\{b_i^2: b_i>0\}=\epsilon^2$ and could yield good overall performance.  Moreover, it is observed that the generated solutions had many tiny values, see Fig.~\ref{fig:ref}. Therefore, we apply a refinement step that keeps the $s$ largest elements in the magnitude of the solution and sets the rest to zeros.   

 \begin{figure}[!th]
\centering
\begin{subfigure}{.495\textwidth}
	\centering
	\includegraphics[width=1\textwidth]{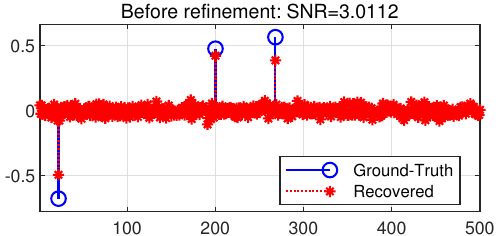}
\end{subfigure}%
\begin{subfigure}{.495\textwidth}
	\centering
	\includegraphics[width=1\textwidth]{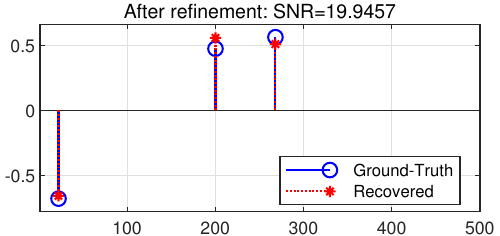}
\end{subfigure} \\\vspace{2mm}
 \begin{subfigure}{.495\textwidth}
	\centering
	\includegraphics[width=1\textwidth]{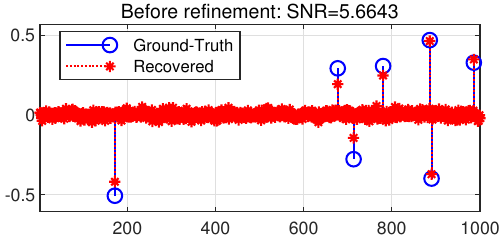}
\end{subfigure}%
\begin{subfigure}{.495\textwidth}
	\centering
	\includegraphics[width=1\textwidth]{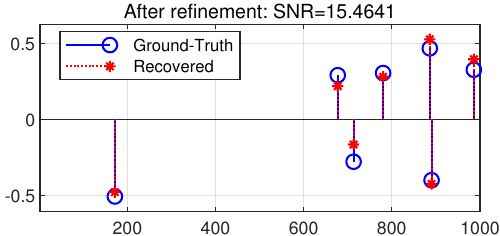}
\end{subfigure} 
\caption{Refinement of the solution.}\vspace{-3mm}
\label{fig:ref}
\end{figure}

\begin{figure}[H]  
\centering
\begin{subfigure}{.325\textwidth}
	\centering
	\includegraphics[width=1\linewidth]{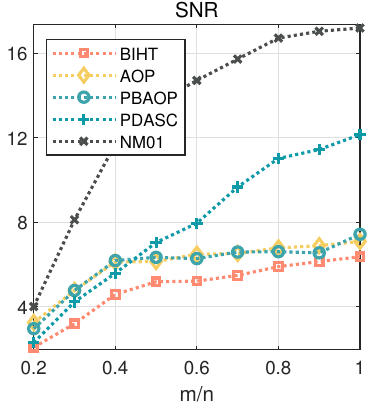}
\end{subfigure}
\begin{subfigure}{.325\textwidth}
	\centering
	\includegraphics[width=1\linewidth]{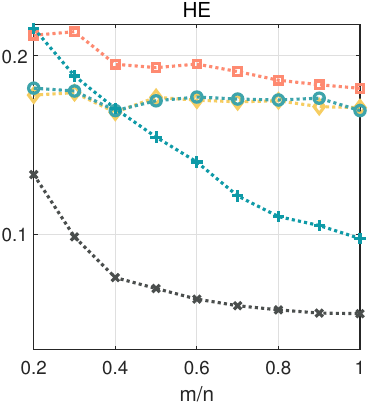}
\end{subfigure}
\begin{subfigure}{.325\textwidth}
	\centering
	\includegraphics[width=1\linewidth]{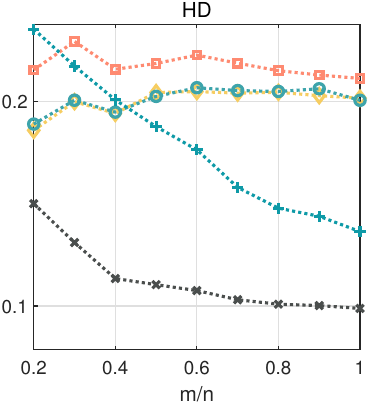}
\end{subfigure}\\\vspace{3mm}

\begin{subfigure}{.325\textwidth}
	\centering
	\includegraphics[width=1\linewidth]{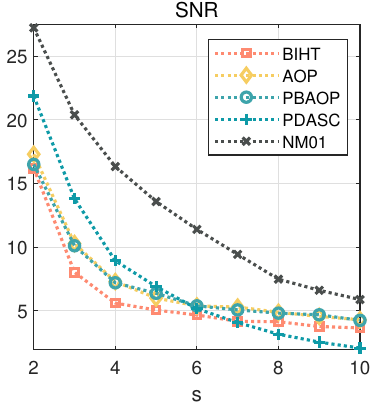}
\end{subfigure}
\begin{subfigure}{.325\textwidth}
	\centering
	\includegraphics[width=1\linewidth]{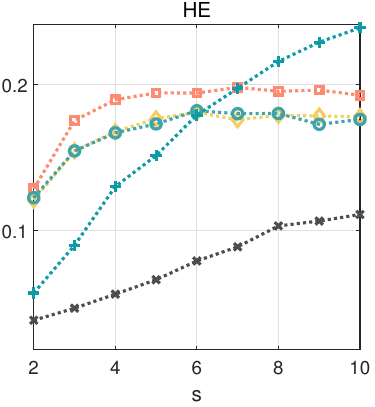}
\end{subfigure}
\begin{subfigure}{.325\textwidth}
	\centering
	\includegraphics[width=1\linewidth]{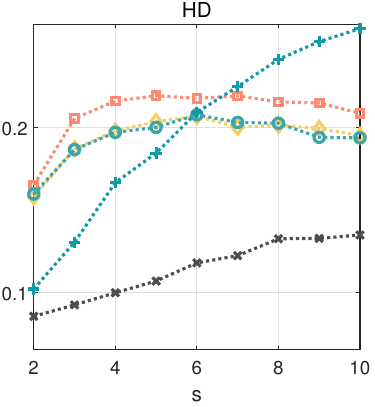}
\end{subfigure}\\ \vspace{3mm}

\begin{subfigure}{.325\textwidth}
	\centering
	\includegraphics[width=1\linewidth]{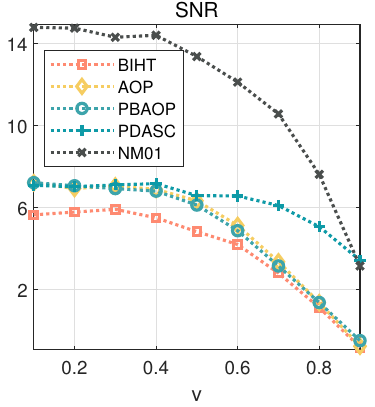}
\end{subfigure}
\begin{subfigure}{.325\textwidth}
	\centering
	\includegraphics[width=1\linewidth]{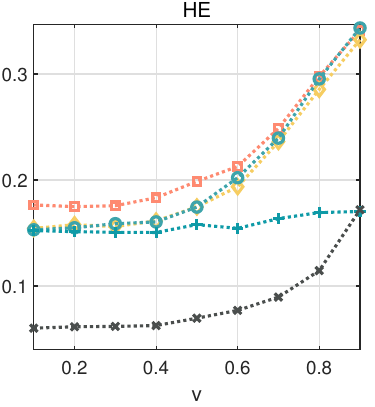}
\end{subfigure}
\begin{subfigure}{.325\textwidth}
	\centering
	\includegraphics[width=1\linewidth]{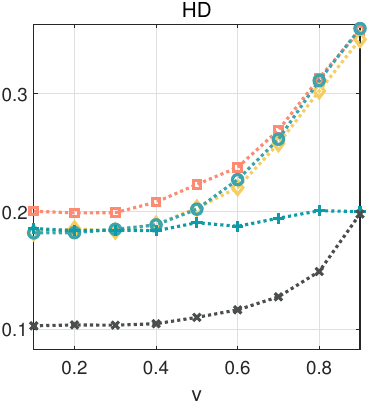}
\end{subfigure}\\\vspace{3mm}

\begin{subfigure}{.325\textwidth}
	\centering
	\includegraphics[width=1\linewidth]{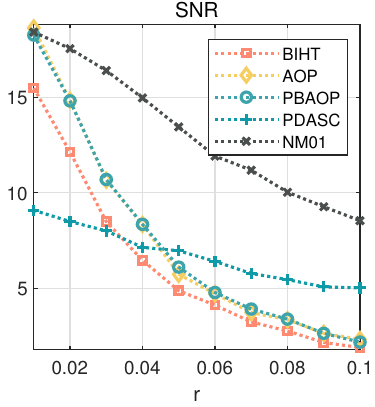}
\end{subfigure}
\begin{subfigure}{.325\textwidth}
	\centering
	\includegraphics[width=1\linewidth]{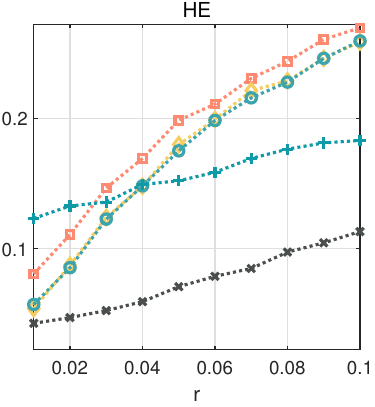}
\end{subfigure}
\begin{subfigure}{.325\textwidth}
	\centering
	\includegraphics[width=1\linewidth]{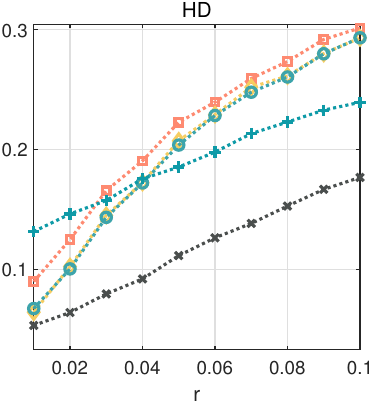}
\end{subfigure}

\caption{Effects of $m$, $s$, $v$ and  $r$   for \cref{ex:cs-cor}.}
\label{fig:ex2-cs-ind-r}\vspace{-3mm}
\end{figure}

Four leading solvers are selected for comparison. They are {\tt PDASC}\footnote{\url{http://jszy.whu.edu.cn/jiaoyuling/en/lwcg/1349484/content/54893.htm\#lwcg}}\cite{huang2018robust}, {\tt BIHT} \footnote{\url{https://laurentjacques.gitlab.io/publication/}}\cite{jacques2013robust}, {\tt AOP} \footref{aop}\cite{yan2012robust}  and {\tt PBAOP}  \footnote{\url{http://www.esat.kuleuven.be/stadius/ADB/huang/downloads/1bitCSLab.zip}\label{aop}}\cite{huang2018pinball},  where the last three require to specify the true sparsity level $s$,  and the last two also need a flipping ratio $L$.  As in \cite{yan2012robust},  we choose $L={\tt HD}$, where {\tt HD} is the hamming distance  generated by {\tt BIHT}.  We also apply the refinement step to {\tt PDASC} so that all five methods produce $s$-sparse solutions.   Finally, all methods start with $\bfx^0=0$  and their solutions are normalized to have a unit length. 
\subsubsection{Comparison} 

We now apply the five methods to solve \cref{ex:cs-cor}  under different scenarios. For each scenario, we report average results over $200$ instances if $n\leq1000$ and $20$ instances otherwise.  For small scale instances, we set five parameters as $(m,n,s,v,r)=(500,250,5,0.5,0.05)$.   To see the effect of each of these parameters,  we tested one parameter while the others being fixed. 

\begin{itemize}
\item  {Effect of $m\in\{0.1,0.2,\ldots,1\}n$}. We note that the bigger $m$ enables the better performance, since more samples are available to recover the signal. It can be clearly seen from Fig.~\ref{fig:ex2-cs-ind-r} that  {\tt NM01} gets the largest {\tt SNR} and the smallest {\tt HD} and {\tt HE}, leading to a better performance than the others.

\item {Effect of $s\in\{2,3,\ldots,10\}$}. The three sub-figures in the second row of   Fig.~\ref{fig:ex2-cs-ind-r} indicate that it is getting more difficult to recover the ground truth signal when $s$ increases.  In comparison with other methods, {\tt NM01} delivers the best recoveries as it achieves the highest {\tt SNR} and the smallest {\tt HD} and {\tt HE}. 

\item {Effect of $v\in\{0.1,0.2,\ldots,0.9\}$}.  The third row sub-figures in  Fig.~\ref{fig:ex2-cs-ind-r} demonstrate that the bigger values of $v$ degrade the performance of each method, because  each pair of two rows of $A_0$ is more correlated with increasing $v$.  
It is observed that {\tt NM01} delivers the best results when $v<0.9$.
  
\item {Effect of  $r\in\{0.02,0.04,\ldots,0.2\}$}.   As expected, the bigger $r$ is (i.e., the more signs are flipped), the harder the recovery is. This can be seen in the sub-figure in the last row of Fig.~\ref{fig:ex2-cs-ind-r}. 
{\tt NM01} outperforms the others.

\item {Effect of $n\in\{2000,4000,\ldots,10000\}$}. For the higher dimensional instances, we fix $m=n/2,  s=5n/1000,  v=0.5$ and  $r=0.05$. We record the average results in   \cref{fig:ex2-cs-ind-n} where  {\tt NM01} achieves the most desirable recovery accuracy. For the computational time,  the other methods are naturally expected to run super-fast since they belong to the family of greedy methods that exploit the sparse structure of the solutions. Nevertheless, {\tt NM01} is relatively competitive in terms of the computational speed. 

\end{itemize}

\begin{table}[!th]\vspace{-2mm}
	\renewcommand{\arraystretch}{.9}\addtolength{\tabcolsep}{1pt}
	\caption{Effect of the higher $n$ for \cref{ex:cs-cor}.}\vspace{-2mm}
	\label{fig:ex2-cs-ind-n}
	\begin{center}
		\begin{tabular}{l|ccccc  c cccccc|}
			\hline
   	&	 {\tt BIHT} 	&	  {\tt AOP} 	&	   {\tt PBAOP}     	&	  {\tt PDASC} 	&	  {\tt NM01} 	&	&	 {\tt BIHT} 	&	  {\tt AOP} 	&	   {\tt PBAOP}     	&	  {\tt PDASC} 	&	  {\tt NM01} 	\\\cline{2-12}
 $n$  	&\multicolumn{5}{c}{{\tt SNR}}&&\multicolumn{5}{c}{{\tt TIME}}\\\cline{1-6}\cline{8-12}
$2000$	&	7.438 	&	6.159 	&	6.960 	&	8.023 	&	11.37 	&	&	0.034 	&	0.414 	&	0.176 	&	0.061 	&	0.170 	\\
$4000$	&	6.509 	&	6.791 	&	6.843 	&	5.545 	&	10.96 	&	&	0.276 	&	1.832 	&	0.773 	&	0.255 	&	0.880 	\\
$6000$	&	7.008 	&	7.014 	&	6.967 	&	3.792 	&	10.02 	&	&	0.803 	&	4.149 	&	2.062 	&	0.636 	&	1.884 	\\
$8000$	&	7.357 	&	7.436 	&	7.225 	&	3.346 	&	10.01 	&	&	1.466 	&	7.279 	&	3.778 	&	1.186 	&	3.715 	\\
$10000$	&	7.841 	&	7.726 	&	7.882 	&	1.489 	&	9.915 	&	&	2.414 	&	11.76 	&	6.777 	&	2.081 	&	5.675 	\\
\hline
  	&\multicolumn{5}{c}{{\tt HE}}&&\multicolumn{5}{c}{{\tt HD}}\\\cline{2-6}\cline{8-12}
$2000$	&	0.201 	&	0.204 	&	0.206 	&	0.180 	&	0.129 	&	&	0.170 	&	0.176 	&	0.175 	&	0.145 	&	0.091 	\\
$4000$	&	0.207 	&	0.203 	&	0.201 	&	0.226 	&	0.125 	&	&	0.177 	&	0.174 	&	0.171 	&	0.198 	&	0.087 	\\
$6000$	&	0.203 	&	0.204 	&	0.206 	&	0.271 	&	0.134 	&	&	0.173 	&	0.174 	&	0.176 	&	0.247 	&	0.097 	\\
$8000$	&	0.205 	&	0.202 	&	0.202 	&	0.285 	&	0.133 	&	&	0.174 	&	0.171 	&	0.172 	&	0.262 	&	0.094 	\\
$10000$	&	0.200 	&	0.201 	&	0.197 	&	0.330 	&	0.135 	&	&	0.168 	&	0.169 	&	0.165 	&	0.312 	&	0.097 	\\
\hline
		\end{tabular}
	\end{center}
\end{table}

\begin{remark} ({\bf On nonsingularity of the Jacobian matrix.})
 {
We finish this section by discussing the important issue of nonsingularity
of the (smoothing) Jacobian matrix $\nabla F_{\mu_k} (\bfw^k, T_k)$
used in  \Cref{Alg-NM01}. 
As pointed out in Introduction, its nonsingularity is equivalent to the
the nonsingularity of $M_k := \nabla^2 f(\bfx^k) + A_{T_k}^\top A_{T_k}/\mu_k$.
If $\nabla^2 f(\bfx^k)$ is positive definite, then $M_k$ is always nonsingular.
This is the case for the SVM problems tested. 
For the problem of 1-bit compressed sensing, we let
\[
  C_* := \nabla^2 f(\bfx^*) = \diag\left\{
    \frac{q\left[\varepsilon^2-(1-q)(x^*_i)^2\right]}{(\varepsilon^2+(x^*_i)^2)^{2-q/2}}, \ \ i=1, \ldots, n
  \right\},
\]
where $\bfx^*$ is the limit of the sequence $\{\bfx^k\}$. 
Suppose $C_*$ is nonsingular, then $C_k := \nabla^2 f(\bfx^k)$ is also nonsingular
when $\bfx^k$ is close to $\bfx^*$. The Woodbury matrix identity implies
that the nonsingularity of $M_k$ is equivalent to that of
the matrix 
\[
\widehat{M}_k := I + \underbrace{({1}/{\mu_k}) A_{T_k} C_k^{-1} A_{T_k}^\top}_{=: \Delta_k}.
\]
Since $\mu_k$ converges to $0$, $\Delta_k$ cannot have $(-1)$ as its eigenvalue when $\bfx^k$ is close to $\bfx^*$ and hence $M_k$ is always nonsingular.
To slightly generalize the above argument, as long as $\Delta_k$ does not
have $(-1)$ among its eigenvalues, $\widehat{M}_k$ (hence $M_k$) is always nonsingular. And the chance for $\Delta_k$ to have $(-1)$ as its eigenvalue is
extremely small in general.
This is what we experienced in our test. 
The argument above does raise the question how to ensure the nonsingularity.
It comes back to the globalization issue of  \Cref{Alg-NM01}.
Our proposal is to use a gradient method whenever singularity becomes an issue.
We leave this to future research.
}	
\end{remark}

\section{Conclusion} \label{Sec-Conclusion}

Optimizing the $0/1$-loss function has been a challenging task for several decades, and few optimality conditions or theoretical convergence guarantees have been established for most of the $0/1$-loss function minimizations.
This paper is the first to develop Newton's method with guaranteed quadratic convergence.
 This has come a long way by first proposing a $P$-stationarity condition that leads to stationarity equations, and then establishing the desired convergence results with help of very technical control over the growth of residue equations. The excellent numerical performance of the proposed method for solving the SVM and 1-bit CS problems indicate that it might work well for other related applications.  We strongly feel that the techniques developed in this paper can be extended to a more general case,  where $A\bfx+\bfb$ in \eqref{ell-0-1} is replaced by some non-linear functions.

\section*{Acknowledgements}
We would like to thank both the referees for their detailed comments that have
helped to 
improve the quality of the paper. In particular, we thank one referee for suggesting the current title and the other for pointing out the link of
the proposed algorithm to the primal-dual active-set algorithms extensively 
studied among semi-smooth Newton methods.

\bibliographystyle{siamplain}
\bibliography{references}

\begin{thebibliography}{10}

\bibitem{bajgier1982experimental}
{\sc S.~M. Bajgier and A.~V. Hill}, {\em An experimental comparison of
  statistical and linear programming approaches to the discriminant problem},
  Decision Sciences, 13 (1982), pp.~604--618.

\bibitem{Beck13}
{\sc A.~Beck and Y.~C. Eldar}, {\em Sparsity constrained nonlinear
  optimization: Optimality conditions and algorithms}, SIAM Journal on
  Optimization, 23 (2013), pp.~1480--1509.

\bibitem{beck2016minimization}
{\sc A.~Beck and N.~Hallak}, {\em On the minimization over sparse symmetric
  sets: projections, optimality conditions, and algorithms}, Mathematics of
  Operations Research, 41 (2016), pp.~196--223.

\bibitem{ben2003difficulty}
{\sc S.~Ben-David, N.~Eiron, and P.~M. Long}, {\em On the difficulty of
  approximately maximizing agreements}, Journal of Computer and System
  Sciences, 66 (2003), pp.~496--514.

\bibitem{BB08}
{\sc P.~T. Boufounos and R.~G. Baraniuk}, {\em 1-bit compressive sensing}, in
  2008 42nd Annual Conference on Information Sciences and Systems, IEEE, 2008,
  pp.~16--21.

\bibitem{brooks2011support}
{\sc J.~P. Brooks}, {\em Support vector machines with the ramp loss and the
  hard margin loss}, Operations Research, 59 (2011), pp.~467--479.

\bibitem{brooks2010analysis}
{\sc J.~P. Brooks and E.~K. Lee}, {\em Analysis of the consistency of a mixed
  integer programming-based multi-category constrained discriminant model},
  Annals of Operations Research, 174 (2010), pp.~147--168.

\bibitem{carrizosa2010binarized}
{\sc E.~Carrizosa, B.~Martin-Barragan, and D.~R. Morales}, {\em Binarized
  support vector machines}, INFORMS Journal on Computing, 22 (2010),
  pp.~154--167.

\bibitem{chang2011libsvm}
{\sc C.~C. Chang and C.~J. Lin}, {\em {LIBSVM}: {A} library for support vector
  machines}, ACM Transactions on Intelligent Systems and Technology (TIST), 2
  (2011), pp.~1--27.

\bibitem{chen1998global}
{\sc X.~Chen, L.~Qi, and D.~Sun}, {\em Global and superlinear convergence of
  the smoothing {N}ewton method and its application to general box constrained
  variational inequalities}, Mathematics of computation, 67 (1998),
  pp.~519--540.

\bibitem{CV95}
{\sc C.~Cortes and V.~Vapnik}, {\em Support-vector networks}, Machine Learning,
  20 (1995), pp.~273--297.

\bibitem{dai2016noisy}
{\sc D.~Dai, L.~Shen, Y.~Xu, and N.~Zhang}, {\em Noisy 1-bit compressive
  sensing: models and algorithms}, Applied and Computational Harmonic Analysis,
  40 (2016), pp.~1--32.

\bibitem{evgeniou2000regularization}
{\sc T.~Evgeniou, M.~Pontil, and T.~Poggio}, {\em Regularization networks and
  support vector machines}, Advances in computational mathematics, 13 (2000),
  pp.~1--50.

\bibitem{fan2014primal}
{\sc Q.~Fan, Y.~Jiao, and X.~Lu}, {\em A primal dual active set algorithm with
  continuation for compressed sensing}, IEEE Transactions on Signal Processing,
  62 (2014), pp.~6276--6285.

\bibitem{fan2008liblinear}
{\sc R.~Fan, K.~Chang, C.~Hsieh, X.~Wang, and C.~Lin}, {\em {LIBLINEAR}: {A}
  library for large linear classification}, Journal of Machine Learning
  Research, 9 (2008), pp.~1871--1874.

\bibitem{feldman2012agnostic}
{\sc V.~Feldman, V.~Guruswami, P.~Raghavendra, and Y.~Wu}, {\em Agnostic
  learning of monomials by halfspaces is hard}, SIAM Journal on Computing, 41
  (2012), pp.~1558--1590.

\bibitem{friedman1997bias}
{\sc J.~H. Friedman}, {\em On bias, variance, 0/1 loss, and the
  curse-of-dimensionality}, Data Mining and Knowledge Discovery, 1 (1997),
  pp.~55--77.

\bibitem{han1987non}
{\sc A.~K. Han}, {\em Non-parametric analysis of a generalized regression
  model: the maximum rank correlation estimator}, Journal of Econometrics, 35
  (1987), pp.~303--316.

\bibitem{hastie2009elements}
{\sc T.~Hastie, R.~Tibshirani, and J.~Friedman}, {\em The elements of
  statistical learning: data mining, inference, and prediction}, Springer
  Science \& Business Media, 2009.

\bibitem{hintermuller2002primal}
{\sc M.~Hinterm{\"u}ller, K.~Ito, and K.~Kunisch}, {\em The primal-dual active
  set strategy as a semismooth {N}ewton method}, SIAM Journal on Optimization,
  13 (2002), pp.~865--888.

\bibitem{huang2018robust}
{\sc J.~Huang, Y.~Jiao, X.~Lu, and L.~Zhu}, {\em Robust decoding from 1-bit
  compressive sampling with ordinary and regularized least squares}, SIAM
  Journal on Scientific Computing, 40 (2018), pp.~A2062--A2086.

\bibitem{huang2018pinball}
{\sc X.~Huang, L.~Shi, M.~Yan, and J.~A. Suykens}, {\em Pinball loss
  minimization for one-bit compressive sensing: Convex models and algorithms},
  Neurocomputing, 314 (2018), pp.~275--283.

\bibitem{huang2006smoothing}
{\sc Z.~Huang, D.~Sun, and G.~Zhao}, {\em A smoothing {N}ewton-type algorithm
  of stronger convergence for the quadratically constrained convex quadratic
  programming}, Computational Optimization and Applications, 35 (2006),
  pp.~199--237.

\bibitem{ito2003semi}
{\sc K.~Ito and K.~Kunisch}, {\em Semi-smooth {N}ewton methods for
  state-constrained optimal control problems}, Systems \& Control Letters, 50
  (2003), pp.~221--228.

\bibitem{jacques2013robust}
{\sc L.~Jacques, J.~N. Laska, P.~T. Boufounos, and R.~G. Baraniuk}, {\em Robust
  1-bit compressive sensing via binary stable embeddings of sparse vectors},
  IEEE Transactions on Information Theory, 59 (2013), pp.~2082--2102.

\bibitem{lai2013improved}
{\sc M.~Lai, Y.~Xu, and W.~Yin}, {\em Improved iteratively reweighted least
  squares for unconstrained smoothed $\ell_q$ minimization}, SIAM Journal on
  Numerical Analysis, 51 (2013), pp.~927--957.

\bibitem{LL2007}
{\sc L.~Li and H.-T. Lin}, {\em Optimizing 0/1 loss for perceptrons by random
  coordinate descent}, in 2007 International Joint Conference on Neural
  Networks, IEEE, 2007, pp.~749--754.

\bibitem{liittschwager1978integer}
{\sc J.~Liittschwager and C.~Wang}, {\em Integer programming solution of a
  classification problem}, Management Science, 24 (1978), pp.~1515--1525.

\bibitem{lu2015optimization}
{\sc Z.~Lu}, {\em Optimization over sparse symmetric sets via a nonmonotone
  projected gradient method}, arXiv preprint arXiv:1509.08581,  (2015).

\bibitem{lutkepohl1996handbook}
{\sc H.~L{\"u}tkepohl}, {\em Handbook of matrices}, vol.~1, Wiley Chichester,
  1996.

\bibitem{ma2005regularized}
{\sc S.~Ma and J.~Huang}, {\em Regularized {ROC} method for disease
  classification and biomarker selection with microarray data}, Bioinformatics,
  21 (2005), pp.~4356--4362.

\bibitem{nguyen2013algorithms}
{\sc T.~Nguyen and S.~Sanner}, {\em Algorithms for direct 0-1 loss optimization
  in binary classification}, in International Conference on Machine Learning,
  2013, pp.~1085--1093.

\bibitem{pelckmans2002matlab}
{\sc K.~Pelckmans, J.~Suykens, T.~Gestel, J.~Brabanter, L.~Lukas, B.~Hamers,
  B.~Moor, and J.~Vandewalle}, {\em A {M}atlab/c toolbox for least square
  support vector machines}, ESATSCD-SISTA Technical Report,  (2002),
  pp.~02--145.

\bibitem{qi2000smoothing}
{\sc H.-D. Qi and L.~Liao}, {\em A smoothing {N}ewton method for general
  nonlinear complementarity problems}, Computational Optimization and
  Applications, 17 (2000), pp.~231--253.

\bibitem{RW1998}
{\sc R.~T. Rockafellar and R.~J. Wets}, {\em Variational analysis}, vol.~317,
  Springer Science \& Business Media, 2009.

\bibitem{rubin1997solving}
{\sc P.~A. Rubin}, {\em Solving mixed integer classification problems by
  decomposition}, Annals of Operations Research, 74 (1997), pp.~51--64.

\bibitem{SV1999}
{\sc J.~A. Suykens and J.~Vandewalle}, {\em Least squares support vector
  machine classifiers}, Neural Processing Letters, 9 (1999), pp.~293--300.

\bibitem{tang2014mixed}
{\sc Y.~Tang, X.~Li, Y.~Xu, S.~Liu, and S.~Ouyang}, {\em A mixed integer
  programming approach to maximum margin 0-1 loss classification}, in 2014
  International Radar Conference, IEEE, 2014, pp.~1--6.

\bibitem{ustun2016supersparse}
{\sc B.~Ustun and C.~Rudin}, {\em Supersparse linear integer models for
  optimized medical scoring systems}, Machine Learning, 102 (2016),
  pp.~349--391.

\bibitem{wang2019support}
{\sc H.~Wang, Y.~Shao, S.~Zhou, C.~Zhang, and N.~Xiu}, {\em Support vector
  machine classifier via $l_{0/1}$ soft-margin loss}, IEEE Transactions on
  Pattern Analysis and Machine Intelligence,  (2021).

\bibitem{weisstein2002heaviside}
{\sc E.~W. Weisstein}, {\em Heaviside step function},
  https://mathworld.wolfram.com/,  (2002).

\bibitem{wu2007robust}
{\sc Y.~Wu and Y.~Liu}, {\em Robust truncated hinge loss support vector
  machines}, Journal of the American Statistical Association, 102 (2007),
  pp.~974--983.

\bibitem{xie2019stochastic}
{\sc M.~Xie, Y.~Xue, and U.~Roshan}, {\em Stochastic coordinate descent for 01
  loss and its sensitivity to adversarial attacks}, in 2019 18th IEEE
  International Conference On Machine Learning And Applications, IEEE, 2019,
  pp.~299--304.

\bibitem{yan2012robust}
{\sc M.~Yan, Y.~Yang, and S.~Osher}, {\em Robust 1-bit compressive sensing
  using adaptive outlier pursuit}, IEEE Transactions on Signal Processing, 60
  (2012), pp.~3868--3875.

\bibitem{zhai2013direct}
{\sc S.~Zhai, T.~Xia, M.~Tan, and S.~Wang}, {\em Direct 0-1 loss minimization
  and margin maximization with boosting}, in Advances in Neural Information
  Processing Systems, 2013, pp.~872--880.

\bibitem{ZXQ21}
{\sc S.~Zhou, N.~Xiu, and H.-D. Qi}, {\em Global and quadratic convergence of
  {N}ewton hard-thresholding pursuit}, Journal of Machine Learning Research, 22
  (2021), pp.~1--45.

\end{thebibliography}
\end{document}